\RequirePackage{fix-cm}
\documentclass{svjour3}
\usepackage[utf8]{inputenc}
\usepackage{graphicx}
\usepackage{mathptmx}
\usepackage{bm}
\usepackage{enumitem}
\usepackage{amsmath,amssymb}
\usepackage{amsfonts}
\usepackage[symbol]{footmisc}
\usepackage{hyperref}

\makeatletter
\let\cl@chapter\undefined
\makeatletter
\usepackage{cleveref}

\usepackage{tikz-cd}
\usepackage{xcolor}
\usepackage{colortbl}
\usepackage{mathrsfs}
\usepackage[ruled,vlined]{algorithm2e}
\usepackage{arydshln}
\usepackage{caption}
\usepackage{anyfontsize}

\usepackage{manyfoot}%

\smartqed  

\DeclareMathOperator{\Ran}{Ran}
\DeclareMathOperator{\Ker}{Ker}
\DeclareMathOperator*{\argmin}{argmin}
\newcommand{\cA}{\mathcal{A}}
\newcommand{\cQ}{\mathcal{Q}}
\newcommand{\cI}{\mathcal{I}}
\newcommand{\cN}{\mathcal{N}}
\newcommand{\cG}{\mathcal{G}}
\newcommand{\cT}{\mathcal{T}}

\newcommand{\fD}{\mathfrak{D}}
\newcommand{\fq}{\mathfrak{q}}
\newcommand{\fp}{\mathfrak{p}}
\newcommand{\fr}{\mathfrak{r}}

\newcommand{\bQ}{\bm{Q}}
\newcommand{\bR}{\bm{R}}
\newcommand{\bq}{\bm{q}}
\newcommand{\br}{\bm{r}}
\newcommand{\bw}{\bm{w}}
\newcommand{\bg}{\bm{g}}
\newcommand{\bnu}{\bm{\nu}}
\newcommand{\bmu}{\mu}

\newcommand{\sB}{\mathsf{B}}
\newcommand{\sPi}{\mathsf{\Pi}}

\newcommand{\sT}{\mathsf{T}}
\newcommand{\sS}{\mathsf{S}}
\newcommand{\sI}{\mathsf{I}}

\newcommand{\sbq}{\mathsf{\bq}}

\DeclareMathOperator{\vdiv}{div}
\renewcommand{\div}{\nabla \cdot}
\newcommand{\curl}{\nabla \times}

\numberwithin{equation}{section}

\journalname{ }
\begin{document}

\title{
	Deep learning based reduced order modeling of Darcy flow systems with local mass conservation\thanks{
	This project has received funding from the European Union's Horizon 2020 research and innovation program under the Marie Skłodowska-Curie grant agreement No. 101031434 -- MiDiROM.
	The present research is part of the activities of the project Dipartimento di Eccellenza 2023-2027, funded by MUR. The authors are members of the Gruppo Nazionale per il Calcolo Scientifico (GNCS) of the Istituto Nazionale di Alta Matematica (INdAM).
	}
}



\author{Wietse M. Boon 
	\and Nicola R. Franco
	\and Alessio Fumagalli
	\and Paolo Zunino
}

\institute{MOX, Department of Mathematics, Politecnico di Milano, Piazza Leonardo da Vinci 32, 20133, Milan, Italy,
              \email{wietsemarijn.boon@polimi.it}
}


\date{Received: date / Accepted: date}

\maketitle

\begin{abstract}
    
    We propose a new reduced order modeling strategy for tackling parametrized Partial Differential Equations (PDEs) with linear constraints, in particular Darcy flow systems in which the constraint is given by mass conservation. Our approach employs classical neural network architectures and supervised learning, but it is constructed in such a way that the resulting Reduced Order Model (ROM) is guaranteed to satisfy the linear constraints exactly. The procedure is based on a splitting of the PDE solution into a particular solution satisfying the constraint and a homogenous solution. The homogeneous solution is approximated by mapping a suitable potential function, generated by a neural network model, onto the kernel of the constraint operator; for the particular solution, instead, we propose an efficient spanning tree algorithm. Starting from this paradigm, we present three approaches that follow this methodology, obtained by exploring different choices of the potential spaces: from empirical ones, derived via Proper Orthogonal Decomposition (POD), to more abstract ones based on differential complexes. All proposed approaches combine computational efficiency with rigorous mathematical interpretation, thus guaranteeing the explainability of the model outputs. To demonstrate the efficacy of the proposed strategies and to emphasize their advantages over vanilla black-box approaches, we present a series of numerical experiments on fluid flows in porous media, ranging from mixed-dimensional problems to nonlinear systems. This research lays the foundation for further exploration and development in the realm of model order reduction, potentially unlocking new capabilities and solutions in computational geosciences and beyond.
%
\end{abstract}


\section{Introduction}
\label{sec: introduction}

Numerical simulation of mathematical models based on partial differential equations (PDEs) is an essential component of many engineering applications. In areas such as geomechanics, hydrology, and exploration of mineral resources, Darcy-type PDEs play a fundamental role, modeling the complex phenomenon of fluid flow in porous media \cite{barbeiro2010priori,eberhard2006simulation,close1993natural}. Typically, in the case of stationary problems, these are of the form
\begin{align}
    \label{eq:darcy-type}
    A_{\bq} \bq + \nabla p &= \bg, & 
    \div \bq &= f,
\end{align}
where $\bq$ and $p$ denote the fluid velocity and its pressure field, respectively. Here, $\bg$ and $f$ are problem data, while $A_{\bq}$ is some positive operator, possibly nonlinear.

Up to supplementing the PDE with suitable boundary conditions, problems such as \eqref{eq:darcy-type} can be easily addressed with classical numerical schemes, providing accurate and reliable approximations of the PDE solution. Nonetheless, despite the incredibly flourishing literature of Numerical Analysis for PDEs, the computational cost entailed by solving a boundary value problem such as \eqref{eq:darcy-type} becomes quickly unbearable when considering many-query applications, where one is required to solve the PDE repeatedly for varying configurations of the system (changes in the operator, in the problem data or in the boundary conditions). For example, in the realm of reservoir geomechanics, the complexity of physics and field configurations, coupled with the need for high-fidelity numerical solutions, demands significant computational resources, especially in scenarios that involve numerous queries or simulations \cite{reservoirgeo}. 
These situations are typically encountered when dealing with optimal control and inverse problems, where the PDE depends on multiple parameters, and the goal is either to optimize their values or to retrieve them from raw measurements. This challenge is not isolated to computational geosciences, but is also prevalent in various computational domains where numerical simulations are used to solve complex physical models \cite{chinesta,doi:10.1137/130932715,doi:10.1137/16M1082469}. 

To address these issues, a well-established approach is to replace the expensive calls to the numerical solver with a cheaper, but accurate surrogate, a paradigm that is also known as Reduced Order Modeling (ROM) \cite{hesthaven_pagliantini_rozza_2022,quarteroni2015reduced}. Here, we shall focus on data-driven non-intrusive ROMs based on Deep Learning techniques, which have recently proven extremely flexible and powerful, see, e.g., \cite{kutyniok,rozzarev}. The authors and coworkers have also recently contributed to this study \cite{fresca2021comprehensive,franco2023deep}. The driving idea behind these approaches is to construct the ROM by learning from a trusted high-fidelity solver, the Full Order Model (FOM), whose quality is taken as a ground truth reference.
In practice, the FOM is exploited \emph{offline} to sample a collection of PDE solutions, from which the ROM is left to learn the implicit mapping between the model parameters and the corresponding PDE solutions. Then, after the training phase, the resulting ROM can be queried \emph{online} at any time at a negligible computational cost.
Compared to other, more traditional strategies, such as the Reduced Basis method, Deep Learning based ROMs have the advantage of being completely non-intrusive, which makes them extremely fast online, and intrinsically non-linear, which provides them with the ability to handle both complex and singular behaviors.

However, being completely data-driven, these approaches are completely unaware of the underlying physics, which typically results in ROMs yielding unphysical results. For example, for problems such as \eqref{eq:darcy-type}, a naive data-driven ROM will likely violate the physical constraint regarding mass conservation. That is, given a suitable configuration of the model parameters, it will produce an output $\tilde{\bq}$ that is a good proxy of the ground truth solution $\tilde{\bq}\approx\bq$, but for which $\div\tilde{\bq}\neq f$.
Clearly, there are many engineering applications for which this phenomenon is undesirable or even unacceptable. Although we may tolerate a small error in the PDE solution, we might also demand that some of the physical constraints are satisfied \emph{exactly}. 

In Machine Learning terms, this corresponds to including the physical constraints within the \emph{inductive bias} of the learning algorithm, thus making the model aware of the underlying mathematical structure. The term inductive bias, in fact, denotes the set of pre-existing assumptions, or beliefs, that a model uses to make predictions, effectively allowing it to extrapolate outside of the training set. In the context of data-driven ROMs, the introduction of inductive biases serves to mitigate the phenomenon of \emph{overfitting} \cite{ying2019overview}, and to accelerate the convergence of deep learning models toward the desired solution. When it comes to physical systems, this translates to the potential inclusion of biases based on first principles, such as conservation laws or geometric symmetries \cite{10.5555/3524938.3525235,LEE2020108973,doi.org/10.1007/s10851-022-01114-x}, among others.

Due to their importance in flow and electromagnetic modeling, the aim of satisfying conservation laws has spawned a broad class of deep-learning approaches. An overview is provided in \cite{Hansen2023}, which emphasizes that enforcing constraints through the loss function, as in Physics-Informed Neural Networks (PINNs) \cite{Raissi2019686,rozzarev} and in their conservative variation \cite{jagtap2020conservative,mao2020physics}, generally does not produce satisfactory results.
Pioneering efforts in the development of exact constraint/preserving networks are reported in \cite{boesen2022neural,ruthotto2020deep}, where the authors discuss the integration of auxiliary algebraic trajectory information into neural networks for dynamical systems, drawing insights from differential-algebraic equations and differential equations in manifolds. 
Along the same lines, significant progress has been achieved in \cite{trask2022enforcing}, where the authors introduce the concept of data-driven exterior calculus, a new technique that offers structure-preserving surrogate models with well-posedness guarantees. The method ensures that machine-learned models strictly enforce physics, even in challenging training scenarios, showing promising results in subsurface flows and electromagnetic modeling. Within the same area of study, we also mention the work by Beuclet et al. \cite{beucler2021enforcing}, in which analytical constraints are enforced by mapping into the kernel of the constraint matrix. 

Here, we restrict our attention to linear constraints, with particular emphasis on equations in divergence form, as in \eqref{eq:darcy-type}.
We propose a Deep Learning strategy to model order reduction that combines computational efficiency with rigorous mathematical interpretation, thus guaranteeing the complete explainability of the model outputs. 
The method relies on separating the flux solution into two components: a homogeneous part and a nonhomogeneous part that satisfies the constraint of interest. Then, we approximate the two separately by proceeding as follows. To approximate the homogeneous solution, we combine the action of a neural network architecture, mapping the model parameters onto a suitable potential function, with that of a kernel projector, effectively binding the output within the null space of the constraint operator. Then, at the same time, we exploit an efficient spanning tree algorithm to retrieve the nonhomogeneous part, which is guaranteed to satisfy the linear constraint exactly. 

Our approach is reminiscent of \cite{rave2019locally} in the context of reduced basis methods. There, the observation is made that, when employing Proper Orthogonal Decomposition (POD) over solenoidal snapshots, each basis function of the reduced space is itself solenoidal. A \emph{shift} of the solution is introduced by means of local flux equilibration. We expand on this idea by exploiting a broader class of solenoidal spaces (not necessarily the POD one) and by providing a shift function that can be computed with linear complexity via the spanning tree algorithm.
In this regard, we also highlight that our construction shares many similarities with \cite{beucler2021enforcing}, but further develops the ideas presented therein by considering non-homogeneous constraints and by providing explicit examples of kernel mappings.

The remainder of this article is structured as follows. The next two subsections introduce the model problem and the adopted notational conventions. \Cref{sec:theory} introduces the approach from an abstract perspective and presents the basic assumptions underlying our construction. These concern the three main ingredients of our proposed ROM strategy: i) a right inverse operator, which we discuss in full detail in \Cref{sec: S_I}, ii) a kernel mapping, to which we devote the entirety of \Cref{sec: S0}, and a iii) potential neural network, thoroughly addressed in \Cref{sec: Neural networks}. Combining these three components leads to constraint-preserving neural network methods, effectively presented in \Cref{sec: approaches}, together with several possible implementations. \Cref{sec: Numerical results} contain numerical experiments that illustrate the performance of the proposed methods compared to simple black-box regression models. Finally, the concluding remarks are given in \Cref{sec: Concluding remarks}.

\subsection{Model Problem}
On a Lipschitz domain $\Omega$, we consider stationary flow problems of the form: find the velocity $\bq$ and pressure $p$ that satisfy
\begin{subequations} \label{eq: model problem}
\begin{align}
    A_{\bq} \bq + \nabla p &= \bg, & 
    \div \bq &= f, 
\end{align}
in which $A_{\bq}$ is a 
positive operator (possibly nonlinear), for given boundary conditions, $\bg$ is a vector source term, and $f$ is a mass source term. The boundary conditions are 
\begin{align}
	p &= p_0, & \text{on } &\partial_p \Omega, &
	\bq \cdot \bnu &= q_0, & \text{on } &\partial_q \Omega,
\end{align}
\end{subequations}
with $\bnu$ the outward unit vector normal to $\partial \Omega$ and the disjoint decomposition $\partial_p \Omega \cup \partial_q \Omega = \partial \Omega$ for which we assume that $|\partial_p \Omega| > 0$.
We focus on two exemplary flow regimes that have this structure:
\begin{itemize}[noitemsep, topsep = 0pt]
    \item Darcy flow: $A_{\bq} \bw := \kappa^{-1} \bw$ for some symmetric positive definite tensor $\kappa$.
    \item Darcy-Forchheimer flow: $A_{\bq} \bw := (\kappa_0^{-1} + \kappa_1^{-1} |\bq|)\bw$ for symmetric positive definite tensors $\kappa_0, \kappa_1$.
\end{itemize}

We assume that this is a parameterized partial differential equation and we collect all relevant parameters of $A_{\bq}, B, \bg,$ and $f$ into the vector $\mu \in M$, where $M$ is the parameter space.
When necessary, we will emphasize the dependency on $\mu$ by using a superscript, e.g., $f^\mu$ instead of $f$. 

We assume that the model problem has been discretized and that we seek the solution in finite-dimensional spaces $\bQ \times P$. For our model problem, the chosen spaces are the lowest-order Raviart-Thomas \cite{raviart-thomas-0} pair:
\begin{align}
    \bQ &:= \mathbb{RT}_0(\Omega_h), &
    P &:= \mathbb{P}_0(\Omega_h),
\end{align}
in which $\Omega_h$ is a simplicial tesselation of $\Omega$.


\subsection{Preliminaries and notation}

Given a Hilbert space $\bQ$, let $\bQ'$ denote its dual. For an operator $B: \bQ \to P'$, let $B^*$ denote its adjoint, so that $B^*:P\to\bQ'$. We will use the Sans Serif font to indicate the matrix equivalent of linear operators, that is, $\sB$ for $B$ with its transpose denoted by $\sB^\top$. Let $N_Q:=\dim(\bQ)$ such that $\bq \in \bQ$ can be represented as a vector $\sbq \in \mathbb{R}^{N_Q}$ containing the values of the degrees of freedom. Similarly, let $N_P$, $N_R$, and $N_M$ denote the dimensions of the finite-dimensional spaces $P$, $R$, and $M$, respectively.

\section{A potential-based approach to satisfy linear constraints}
\label{sec:theory}

We first cast the problem in a slightly more general framework. Let $\bQ$ and $P$ be Hilbert spaces. We consider a class of problems that have the following form:
Given functionals $(\bg, f) \in \bQ' \times P'$, find $(\bq, p) \in \bQ \times P$ such that
\begin{subequations} \label{eq: general problem}
\begin{align}
    A_{\bq} \bq - B^* p &= \bg,
    & \text{in }&\bQ',
    \label{eq: general constitutive law}\\
    B \bq &= f,
    & \text{in }&P'.
    \label{eq: general mass conservation}
\end{align}
\end{subequations}
Here, $A_{\bq} : \bQ \to \bQ'$ is a (non-)linear, positive operator, while $B : \bQ \to P'$ is linear and surjective. 
The variable $p$ is the Lagrange multiplier that enforces the constraint.

Our approach relies on a decomposition of $\bq \in \bQ$ into a \emph{particular} solution $\bq_f$ and a \emph{homogeneous} solution $\bq_0$,
\begin{align}
    \bq &= \bq_f + \bq_0, &
    B \bq_f &= f, &
    \bq_0 &\in \bQ_0 := \Ker(B, \bQ).
\end{align}

For the construction of the particular and homogeneous parts of the solution, we assume that the following key tools are available to us:
\begin{subequations}
\begin{enumerate}[label = A\arabic*., ref = A\arabic*]
    \item \label{ass: S_I}
    A computationally efficient operator $S_I : P' \to \bQ$ that is a right-inverse of $B$, i.e. for which
    \begin{align}
        B S_I  = I.
    \end{align}
    \item \label{ass: S_0}
    An auxiliary Hilbert space $\bR$, which we refer to as the \emph{potential space}, and a linear map $S_0: \bR \to \bQ_0$. Then, by definition, we have
    \begin{align}
        BS_0 = 0.
    \end{align}
    \item \label{ass: N}
    A neural network map $\cN : M \to \bR$ that is trained to make the following approximation,
    \begin{align}
        S_0 \cN \mu \approx (I - S_I B)\bq^\mu,
    \end{align}
    where $\bq^\mu \in \bQ$ is the true solution to \eqref{eq: general problem} associated with the values of the corresponding parameters $\mu\in M$.
\end{enumerate}
\end{subequations}

\noindent The following diagram illustrates the connection between the operators and spaces introduced above:
\begin{equation}
\begin{tikzcd}
    M \arrow[r, "\cN"] &
    \bR \arrow[r, "S_0"] &
    \bQ \arrow[r, shift left, "B"] &
    P' \arrow[l, shift left, "S_I"].
\end{tikzcd}
\end{equation}

We are now ready to describe the three-step solution technique, which is a generalization of the procedure presented in \cite{boon2023reduced}:
\begin{enumerate}[label=\arabic*., noitemsep]
    \item Compute the particular solution $\bq_f = S_I f \in \bQ$.
    \item Use the neural network and $S_0$ to compute the homogeneous solution $\tilde \bq_0 = S_0 \cN \mu$.
    \item Set $\tilde \bq = \bq_f + \tilde \bq_0$. The approximation $\tilde p$ can be postprocessed as:
    $$\tilde p
        =
        (S_I ^* B^*) \tilde p
        =
        S_I ^*(A_{\tilde \bq} \tilde \bq - \bg).
    $$
\end{enumerate}

\begin{lemma} \label{lem: satisfaction constraint}
    The approximate solution $\tilde \bq$ satisfies the linear constraint \eqref{eq: general mass conservation}, regardless of the precision of the mapping $\cN$.
\end{lemma}
\begin{proof}
By assumptions \ref{ass: S_I}-\ref{ass: S_0}, we have
$
    B \tilde \bq
    = B (\bq_f + \tilde \bq_0)
    = B (S_I f + S_0 \cN\mu)
    = f.
$
\qed\end{proof}

We dedicate the next sections to 
an in-depth discussion about the possible choices for the three key components of the proposed procedure. That is, we shall discuss how to choose: a right-inverse of $B$ (A1), a potential space supplied with a kernel-mapping (A2), and a suitable neural network architecture (A3).

\section{A right-inverse of the operator \texorpdfstring{$B$}{B}}
\label{sec: S_I}

The first component concerns an operator that can be used to generate a particular solution $\bq_f$ such that $B \bq_f = f$. 
In principle, any right-inverse of $B$ can be chosen: here, we shall focus on those admitting an explicit construction. The main idea is to choose an appropriate operator $E : \bQ \to \bQ$ such that $BE$ is invertible. 
Then, we may let
\begin{align} \label{eq: structure S_I}
    S_I  := E (B E)^{-1},
\end{align}
as it is straightforward to see that $BS_I = I$, as required by assumption \ref{ass: S_I}. Two examples of admissible operators $E$ are considered in the next two subsections.

\subsection{The Moore-Penrose inverse}
\label{sub: mp-inverse}
As a first example, we may set $E = LB^*$ for some $L: \bQ' \to \bQ$. 
However, in this work we require an exact, right-inverse operator to ensure a precise satisfaction of the constraint. 
The system $B L B^*$ therefore needs to be solved to machine precision for each input, which may be too computationally demanding, unless specific tailored solvers are available or if $\Ran(B)$ is of low dimensionality.

\begin{remark}
This structure was explored in \cite{boon2023reduced} by recognizing that the finite volume method using two-point flux approximation is of this form with $\mathsf{L}$ a diagonal matrix. 
\end{remark}



\subsection{A spanning tree solve for conservation equations}
\label{sub: spanning tree}

Next, we focus on the particular case in which $B$ is a divergence operator acting on the Raviart-Thomas space of lowest order, on a simplicial grid. In this case, we leverage the structure of the operator $B$ and the fact that $\mathbb{RT}_0$ has one degree of freedom per face to find an appropriate mapping $E$ for \eqref{eq: structure S_I}.
Our construction is similar in spirit to \cite[Sec. 7.3]{jiranek2010posteriori}.

We now detail the construction of a particular right-inverse $S_I: P' \to \bQ$, which we refer to as a \emph{spanning tree solve}. The idea goes as follows:
\begin{enumerate}[label = \arabic*.]
    \item Find a column of $\sB \in \mathbb{R}^{n \times m}$ that has exactly one entry $b_{i_0 j_0}$. The index $j_0$ corresponds to a facet degree of freedom in $\bQ$ that is on the boundary and $i_0$ is its neighboring element.
    \item Construct a graph $\cG$ with $n$ nodes and $m$ edges that has an incidence matrix with the same sparsity pattern as $\sB$. We emphasize that the graph nodes correspond to the mesh cells and the graph edges correspond to the mesh facets.
    \item Generate a spanning tree $\cT$ of $\cG$ with $i_0$ as its root. For example, we may use the tree generated by a breadth-first search of the graph, which is a computation of linear complexity. Note that $\cT$ contains all nodes of $\cG$ and a subset of its edges.
    \item Let $J$ be the set that contains the $(n - 1)$ edges of $\cG$ that are present in the tree $\cT$, complemented by the edge with index $j_0$.
    Let $\sPi \in \mathbb{R}^{n \times m}$ be the restriction operator on the $n$ facet degrees of freedom in the mesh that correspond to the edges in $J$.
    \item Due to the tree structure, the matrix $\sB \sPi^\sT \in \mathbb{R}^{n \times n}$ is triangular, up to the column/row permutation, and therefore leads to a system that is computationally easy to solve. We are now ready to define
    \begin{align}
        \sS_\sI &:= \sPi^\sT (\sB \sPi^\sT)^{-1}.
    \end{align}
\end{enumerate}

We emphasize that this choice of $S_I$ fits the format \eqref{eq: structure S_I} by setting $E$ as the adjoint of a restriction operator $\Pi$. In turn, Assumption \ref{ass: S_I} is directly satisfied. An example of a spanning tree is given in \Cref{fig:spanningtree}.

\begin{remark}
\label{remark: averaged-tree}
Note that the set of right-inverses of $B$, namely
 \begin{align}
    \{S:P'\to\bQ\;|\;BS=I\},     
 \end{align}
 is a convex set. This observation can be used to generate other candidates for the right-inverse of $B$, starting from the approaches described in Sections \ref{sub: mp-inverse} and \ref{sub: spanning tree}. For example, assume that $S_{I}^{1},\dots,S_{I}^{N_{t}}$ are $N_{t}$ right-inverses of $B$ obtained by relying on $N_{t}$ different spanning trees, as in \Cref{sub: spanning tree}. Then, the ``\emph{average tree solver}''
 \begin{align}
    \label{eq:many trees}
     S_{I}^{\text{avg}}:=\frac{1}{N_{t}}\sum_{i=1}^{N_{t}}S_{I}^{i},
 \end{align}
 also operates as a right-inverse of $B$, thus fulfilling \ref{ass: S_I}. In practice, as we shall discuss in \Cref{sec: Numerical results}, combining multiple trees as in \eqref{eq:many trees} can substantially improve the quality of the model, especially in terms of the approximation of the pressure field. 
\end{remark}

\section{An operator that maps onto the kernel of \texorpdfstring{$B$}{B}}
\label{sec: S0}

The second component of our approach concerns an operator that can be used to compute the homogeneous solution $\bq_0 \in \bQ_0$. We propose three examples for this operator in the following subsections, and remark in each case on their surjectivity.

\begin{figure}
	\centering
	\includegraphics[width=0.5\textwidth]{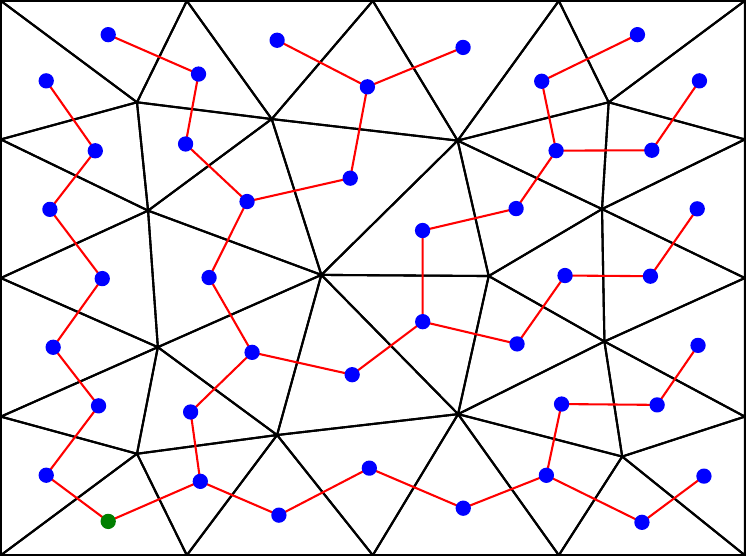}\vspace{0.25cm}
	\caption{A spanning tree $\cT$ superimposed on a triangular, unstructured grid. A particular solution $\bq_f$ is generated that balances the mass source $f$ by progressively solving the conservation equation from the leaves to the root of the tree at the domain boundary, illustrated in green in the bottom left of the figure. Only the degrees of freedom on faces cut by the tree are taken into account.}
	\label{fig:spanningtree}
\end{figure}

\subsection{A choice based on the right-inverse}

We first consider a construction that follows directly from a given operator $S_I$ that satisfies assumption \ref{ass: S_I}. Due to its simplicity, we present its definition together with its admissibility in the next proposition.

\begin{proposition} \label{prop: S0 as I-SB}
	If $S_I$ is chosen according to \ref{ass: S_I}, then the choice
	\begin{subequations}
	\begin{align}
		\bR &:= \bQ, &
		S_0 &:= I - S_I B, 
	\end{align}
	satisfies \ref{ass: S_0}. Moreover, $S_0$ is a projection and
	\begin{align} \label{eq: ker S_0}
		\Ker(S_0) &= \Ran(S_I B).
	\end{align}
	\end{subequations}
\end{proposition}
\begin{proof}
A direct calculation gives us $B S_0 = B(I - S_I B) = B - B = 0$, which implies that $\Ran(S_0) \subseteq \bQ_0$ verifying \ref{ass: S_0}. Next, we note that
\begin{align}
    S_0 \bq_0 &= (I - S_I B) \bq_0 = \bq_0, &
    \forall \bq_0 &\in \bQ_0.
\end{align}
Hence $S_0$ is the identity operator on $\bQ_0$ and is therefore a projection. To prove \eqref{eq: ker S_0}, we first note that
\begin{align}
    S_0 (S_I B) = S_I B - S_I (B S_I) B = 0
\end{align}
which shows that $\Ker(S_0) \supseteq \Ran(S_I B)$. Secondly, for $q \in \Ker(S_0)$, we have $(I - S_I B) q = 0$ 
$\implies$ $q = S_I B q$. In turn, $q \in \Ran(S_I B)$ from which we conclude $\Ker(S_0) \subseteq \Ran(S_I B)$.
\qed\end{proof}

\subsection{Proper orthogonal decomposition}
\label{sub: POD}

We now consider an alternative mapping based on techniques common to reduced basis methods. We assume to be in the setting of model problem \eqref{eq: model problem}, so that $\bQ=\mathbb{RT}_{0}$ and $\|\cdot\|_{\bQ}=\|\cdot\|_{H(\vdiv)}$.
With this setup, we assume that $\{\mu_{i}\}_{i=1}^{N_{s}}$ is a sampling of the parameter space, i.e.\ $\mu_i \in M$ for each index $i$ with $N_s$ the number of samples.
Let $\{\bq^{\mu_{i}}_{0}\}_{i=1}^{N_{s}}$ be the homogeneous parts of the corresponding flux fields. 
For given $n \le N_Q$, we then construct the empirically optimal projector $V_n : \mathbb{R}^{n}\to\bQ_{0}$ that minimizes the reconstruction error:
\begin{align}
    \label{eq:podopt}
    V_n :&=\argmin_{W:\mathbb{R}^{n}\to\bQ_{0}}\;\frac{1}{N_{s}}\sum_{i=1}^{N_{s}}\|\bq^{\mu_{i}}_{0}-WW^{*}\bq^{\mu_{i}}_{0}\|_{\bQ}^{2}\nonumber\\
    &=\argmin_{W:\mathbb{R}^{n}\to\bQ_{0}}\;\frac{1}{N_{s}}\sum_{i=1}^{N_{s}}\|\bq^{\mu_{i}}_{0}-WW^{*}\bq^{\mu_{i}}_{0}\|_{L^{2}}^{2}.
\end{align}
The last equality follows from the fact that $\|\bq\|_{\bQ}=\|\bq\|_{L^{2}}$ for all $\bq\in\bQ_{0}.$
From a practical point of view, this is achieved by means of \emph{Proper Orthogonal Decomposition} (POD), see e.g. \cite[Sec. 6.3]{quarteroni2015reduced}. In particular, the solution to \eqref{eq:podopt} is known in closed form and can be computed as follows. First, let the so-called \emph{snapshots matrix} be defined as
\begin{align} \label{eq: def snapshot matrix}
    \mathsf{U}:=[\mathsf{q}^{\mu_{1}}_{0},\ldots,\mathsf{q}^{\mu_{N_{s}}}_{0}]^{\top}\in\mathbb{R}^{N_{s}\times N_Q}.
\end{align}
Second, let $\mathsf{M} \in \mathbb{R}^{N_Q\times N_Q}$ be the mass matrix representing the $L^{2}$-inner product over $\bQ$, 
that is, the matrix that satisfies
\begin{align}
    \mathsf{u}^\top \mathsf{M} \mathsf{v}
     = 
    \langle\bm{u}, \bm{v}\rangle_{L^{2}}
\end{align}
for all $\bm{u}, \bm{v}\in\bQ$ with vector representations $\mathsf{u}, \mathsf{v} \in \mathbb{R}^{N_Q}$. We then perform an eigenvalue decomposition to form matrices $\tilde{\mathsf{U}}, \mathsf{\Lambda} \in \mathbb{R}^{N_{s} \times N_{s}}$ that satisfy
\begin{align}
	\mathsf{U}\mathsf{M}\mathsf{U}^{\top}=\tilde{\mathsf{U}}\mathsf{\Lambda}\tilde{\mathsf{U}}^{\top},
\end{align}
with $\mathsf{\Lambda}=\text{diag}(\lambda_{1},\dots,\lambda_{N_{s}})$ containing the eigenvalues $\lambda_{1}\ge\dots\ge\lambda_{N_{s}}\ge0.$ Let now $\tilde{\mathsf{U}}_{n}\in\mathbb{R}^{n\times N_{s}}$ be the matrix obtained by extracting the first $n$ rows of $\tilde{\mathsf{U}}.$ Similarly, let \begin{align}
	\mathsf{\Lambda}_{n}^{-1/2}:=\text{diag}\left(\frac{1}{\sqrt{\lambda_{1}}},\dots,\frac{1}{\sqrt{\lambda_{n}}}\right).
\end{align}
Finally, the POD projector $V_n: \mathbb{R}^n \to \bQ_{0}$ is defined according to its matrix representation as
\begin{align} \label{eq: def POD projection}
	\mathsf{V_n}:= \left(\mathsf{\Lambda}^{-1/2}_{n}\tilde{\mathsf{U}}_{n}\mathsf{U} \right)^{\top}.
\end{align}

\begin{proposition} \label{prop: S0 POD}
    Let $n \le N_Q$ and let $V_{n}$ be as in \eqref{eq: def POD projection}. Then,
	\begin{align}
		\bR &:= \mathbb{R}^n, &
		S_0 &:= V_n, 
	\end{align}
	satisfies assumption \ref{ass: S_0}.
\end{proposition}
\begin{proof}
	By definitions \eqref{eq: def POD projection} and \eqref{eq: def snapshot matrix}, the range of $V_n$ is a subspace of $\operatorname{span} \{\bq^{\mu_{i}}_{0}\}_{i=1}^{N_{s}}$. 
    Consequently, $\Ran(V_n) \subseteq \bQ_0$.
\qed\end{proof}

It is important to note that this choice of $S_0$ is generally not surjective on the kernel $\bQ_0$. However, the range of $V_n$ 
can be a good approximation of $\bQ_0$ for sufficiently large $n$. 

\subsection{Differential complexes}
\label{sub: diff complex}
Our final example concerns the case in which $B$ corresponds to a differential operator. In particular, if a $d^k: \bQ \to P$ exists such that
\begin{align} \label{eq: def B dk}
	\langle B \bq, \tilde p \rangle = \langle d^k \bq, \tilde p \rangle_P
\end{align}
and $d^k$ is an operator in a \emph{differential complex}. Recall that a differential complex is a sequence of spaces $\Lambda^k$ connected by operators $d^k : \Lambda^k \to \Lambda^{k + 1}$ such that $d^k d^{k - 1} = 0$ for all $k$. Identifying $P := \Lambda^{k + 1}$, $\bQ := \Lambda^k$, and $\bR := \Lambda^{k - 1}$, the differential complex is illustrated as
\begin{equation}
\begin{tikzcd} \label{eq: diff complex}
    \ldots \arrow[r, "d^{k-2}"] &
    \bR \arrow[r, "d^{k-1}"] &
    \bQ \arrow[r, "d^k"] &
    P \arrow[r, "d^{k + 1}"] &
    \ldots
\end{tikzcd}
\end{equation}
The following result is directly obtained from the definitions.

\begin{proposition} \label{prop: S0 diffcomp} 
	Let $d^k$ be such that \eqref{eq: def B dk} holds and let \eqref{eq: diff complex} be a differential complex with $\bQ = \Lambda^k$. Then the choice
	\begin{align}
		\bR &:= \Lambda^{k - 1}, &
		S_0 &:= d^{k - 1},
	\end{align}
	satisfies assumption \ref{ass: S_0}.
\end{proposition}

In our model problem \eqref{eq: model problem}, we have the divergence operator $d^k \bq = \div \bq$ and $\bQ = \mathbb{RT}_0$. In the case in which $\Omega$ is three-dimensional, we can use the theory of finite element exterior calculus \cite{arnold2006finite,arnold2018finite} and consider the following differential complex:
\begin{align}
	\mathbb{N}_0(\Omega_h) \xrightarrow{\curl} \mathbb{RT}_0(\Omega_h) \xrightarrow{\div} \mathbb{P}_0(\Omega_h).
\end{align}
in which $\mathbb{N}_0$ is the N\'ed\'elec edge element of the first kind \cite{nedelec1980mixed}.
We emphasize that this is a differential complex due to the vector calculus identity $\div \curl \br = 0$ for all sufficiently regular $\br$.
An analogous differential complex exists in the two-dimensional case, in which the nodal Lagrange space $\mathbb{L}_1(\Omega_h)$ replaces $\mathbb{N}_0(\Omega_h)$ and we have the rotated gradient $\nabla^\perp r :=  [\partial_2 r,\ - \partial_1 r]^\top$ as the curl.

We remark that $S_0$ as chosen in \Cref{prop: S0 diffcomp} is surjective onto $\bQ_0$ if the domain is contractible. Otherwise, e.g.\ if the domain has holes, the homogeneous solution may not be fully computable in this way. Extra steps then need to be taken to find the part of the solution that is in $\Ker(d^k) / \Ran(d^{k - 1})$, also known as the cohomology space. In this work, we focus on the simplified case with zero cohomology.

\section{A neural network model that maps onto the potential space} 
\label{sec: Neural networks}

The third and final component of our approach consists of an efficient neural network $\cN : M \to \bR$ that can produce a suitable potential $\br=\br^\mu$ for any given parameter instance $\mu$ (cf. assumption \ref{ass: N}). 
We recall that the method proposed in \Cref{sec:theory} constructs an approximation $(\tilde \bq, \tilde p)$ for a given $\mu$ in the following manner:
\begin{align}
\label{eq:tilde}
	\tilde \bq &= S_I f^\mu + S_0 \cN(\mu), &
	\tilde p &= S_I^* (A_{\tilde \bq} \tilde \bq - g^\mu).
\end{align}

Since we have already defined the linear operators $S_0$ and $S_I$ at this stage, we can directly train the neural network such that the map $\mu \mapsto (\tilde \bq, \tilde p)$ approximates the solution map $\mu \mapsto (\bq^\mu, p^\mu)$. To this end, we note that the following holds true.

\begin{lemma}
    \label{lemma: pressure}
    Assume that: 
    \begin{enumerate}[label = \roman*)]
    \item The map $\cA : \bq\mapsto A_{\bq}$ is locally Lipschitz continuous.
    \item The solution map $\cQ : \mu\mapsto\bq^{\mu}$ is continuous.
    \item The source map $F : \mu\mapsto f^{\mu}$ is continuous. 
    \end{enumerate}
    Then, there exists a constant $C=C(\cA,\cQ,F,S_{I},S_{0},\cN)>0$ such that, for all $\mu\in M$ one has
    \begin{align}
    \|p - \tilde p\|_{P}\le C\|\bq-\tilde\bq\|_{\bQ},
    \end{align}
    where $\bq=\bq^{\mu}$ and $p=p^{\mu}$, while $\tilde{\bq}$ and $\tilde{p}$ are as in \eqref{eq:tilde}. 
\end{lemma}

\begin{proof}
    Let $\bQ_{M}:=\{\bq^{\mu'}\;|\;\mu'\in M\}\subset\bQ$ be the solution manifold associated with the flux component. We note that, by continuity of the solution map, $\bQ_{M}$ is compact. Since locally Lipschitz continuous maps are Lipschitz over compact sets, we may denote the Lipschitz constant of $\cA$ over $\bQ_{M}$ by $L_A$. 

    We have,
    \begin{align*}
    \|p - \tilde p\|_{P} =
    \| S_{I}^{*} (A_{\bq}\bq - A_{\tilde \bq}\tilde{\bq})\|_{P} &\le
    \| S_{I} \|\ \|A_{\bq}\bq - A_{\tilde \bq}\tilde{\bq}\|_{\bQ'} \\
    &\le
    \| S_{I} \| \left(\|A_{\bq}\bq - A_{\bq}\tilde{\bq}\|_{\bQ'} + \|A_{\bq}\tilde{\bq} - A_{\tilde \bq}\tilde{\bq}\|_{\bQ'}\right)\\
    &\le\| S_{I} \| \left(\|A_{\bq}\|\ \|\bq - \tilde{\bq}\|_{\bQ} + \|A_{\bq} - A_{\tilde \bq}\|\ \|\tilde{\bq}\|_{\bQ}\right)\\
    &\le\| S_{I} \| \left(\|A_{\bq}\|\ \|\bq - \tilde{\bq}\|_{\bQ} +  L_A \|\bq-\tilde \bq\|\ \|\tilde{\bq}\|_{\bQ}\right)
    ,\end{align*}
    in which $\| S_I \|$ and $\| A_{\bq} \|$ refer to the operator norms.

    We now note that, since the maps $S_{I},F,S_{0},\mathcal{N}$ are continuous, so is the map $\mu\mapsto S_{I}f^{\mu}+S_{0}\cN(\mu)$. The compactness of $M$ then implies that some $c'>0$ exists such that $\|\tilde\bq\|_{\bQ}\le c'$, regardless of $\mu\in M$. Similarly, we can bound $\|A_{\bq}\|$ by some $c''>0$. It follows that
    \begin{align}
        \|p - \tilde p\|_{P}\le \| S_{I} \| (c''+Lc')\|\bq-\tilde\bq\|_{\bQ},
    \end{align}
    as claimed.
\qed\end{proof}

In other words, Lemma \ref{lemma: pressure} shows that it suffices to train the potential network $\cN$ to produce a reliable approximation of the flux, without directly taking care of the pressure. This allows us to avoid reassembling $A_{\tilde\bq}$ during the training phase.
In light of these considerations, one possibility is to directly train the neural network model $\cN$ by minimizing the following loss function
\newcommand{\loss}{\mathscr{L}}
\begin{align}
    \label{eq:loss1}
    \loss(\cN)=\frac{1}{N_{s}}\sum_{i=1}^{N_{s}}\left\|\bq^{\mu_{i}}-\left[S_I f^\mu_{i} + S_0 \cN (\mu_{i})\right]\right\|_{\bQ}^{2}&=\frac{1}{N_{s}}\sum_{i=1}^{N_{s}}\left\|(I-S_{I}B)\bq^{\mu_{i}} - S_0 \cN (\mu_{i})\right\|_{\bQ}^{2} \nonumber\\
    &=\frac{1}{N_{s}}\sum_{i=1}^{N_{s}}\left\|\bq_{0}^{\mu_{i}} - S_0 \cN (\mu_{i})\right\|_{\bQ}^{2},
\end{align}
where $\bq_{0}^\mu:=(I-S_{I}B)\bq^{\mu}$, in line with our adopted notation. Here, $\{\mu_{i}, \bq^{\mu_{i}}\}_{i=1}^{N_{s}}\subset M\times\bQ$ is a suitable collection of training data, previously generated via the FOM. 

In practice, optimizing \eqref{eq:loss1} can become problematic if the potential space $\bR$ is high dimensional, as the network is required to produce a large number of outputs.
In that case, one may enhance the training by introducing a \emph{latent regularization}. This idea was first proposed by Fresca et al. in the so-called DL-ROM approach \cite{fresca2021comprehensive}, where the network architecture is decomposed in two components, $\phi:M\to\mathbb{R}^{n}$ and $\Psi:\mathbb{R}^{n}\to\bR$, so that $\cN = \Psi\circ\phi$. This splitting emphasizes the existence of a latent space, here denoted by $\mathbb{R}^{n}$, which lives in between the several hidden states of the overall architecture $\cN$. 
Then, the DL-ROM approach proposes to replace \eqref{eq:loss1} with
\newcommand{\reg}{\textnormal{reg}}
\begin{align}
    \label{eq:loss2}
    \loss_{\reg}(\phi,\Psi,\Psi')=\frac{1}{N_{s}}\sum_{i=1}^{N_{s}}\left\|\bq_{0}^{\mu_{i}} - S_0 (\Psi\circ\phi) (\mu_{i})\right\|_{\bQ}^{2} + \lambda\frac{1}{N_{s}}\sum_{i=1}^{N_{s}}\left\|\Psi'(\bq_{0}^{\mu_{i}}) -\phi(\mu_{i})\right\|_{\mathbb{R}^{n}}^{2},
\end{align}
where $\lambda>0$ is a suitable regularization parameter, while $\Psi':\bQ\to\mathbb{R}^{n}$ is an auxiliary network that is used to enforce a connection between the final output and the latent representation of the model \cite{franco2023latent}. In summary, this approach requires the training of three networks: $\phi$, $\Psi$ and $\Psi'$.

\begin{remark}
    In the DL-ROM literature, e.g. \cite{fresca2021comprehensive,franco2023deep}, the maps $\Psi'$ and $\Psi$ are usually referred to as \emph{encoder} and \emph{decoder}. This is because the same idea can be explained using the concept of \emph{autoencoders}. To see this, note that the regularization term makes the networks act so that $\Psi'(\bq_{0}^\mu)\approx\phi(\mu)$ and $\bq_{0}^{\mu}\approx S_{0}\Psi(\phi(\mu))$. It then follows that
    \begin{equation}
        \label{eq:reconstruction}
        \bq^{\mu}_{0}\approx S_{0}(\Psi(\Psi'(\bq_{0}^{\mu}))),
    \end{equation}
    meaning that the map $S_{0}\circ\Psi\circ\Psi'$ can operate as an autoencoder, effectively compressing and reconstructing the homogenous part of flux solutions (this is why we chose to represent the latent dimension with the letter $n$, to highlight its parallelism with the reduced dimension in the POD approach, cf. Section \ref{sub: POD}). Here, however, we consider the DL-ROM approach as a regularization strategy for $\cN:=\Psi\circ\phi$, as talking about autoencoders may generate confusion between the flux space $\bQ$ and the potential space $R$. Furthermore, for our purposes, we are not interested in the quality of the reconstruction error \eqref{eq:reconstruction}.
\end{remark}

\section{Putting things together: neural network methods respecting linear constraints}
\label{sec: approaches}


We are now ready to combine the three ingredients and present the overall workflow of our approach. Based on the strategies described in the previous sections, we shall derive three implementations, each with their own benefits and limitations. For each, we choose the operator $S_{I}$ to be the spanning tree solve from \Cref{sub: spanning tree}, or an averaged variation (cf. Remark \ref{remark: averaged-tree}). Thus, the main differences lie in the choice of $S_{0}$, cf. \Cref{sec: S0}, and in the corresponding neural network architectures and training strategies.
We highlight these in more detail in the following subsections.





\begin{figure}[tb]
	\centering
	\includegraphics[width=0.9\textwidth]{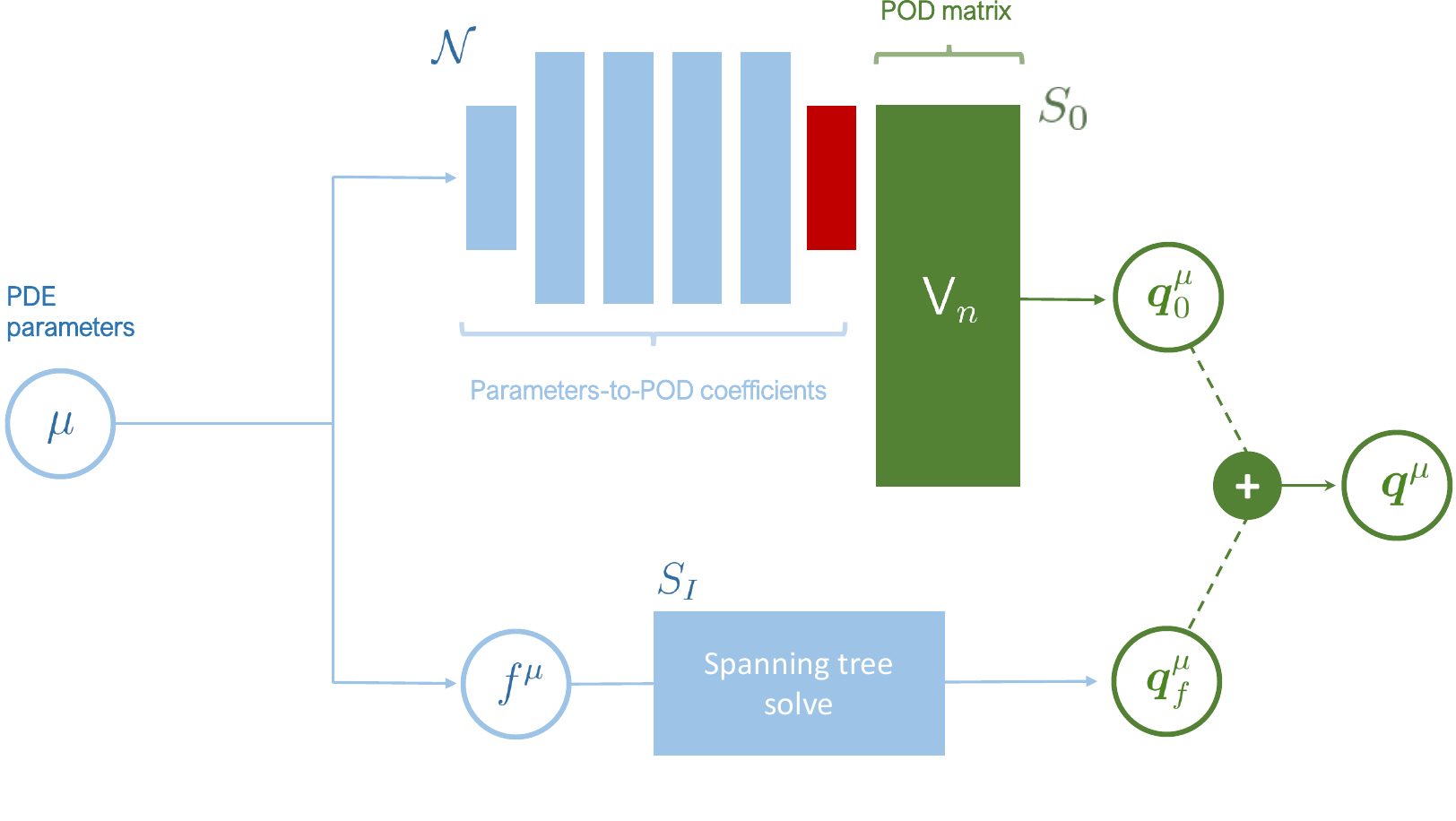}
	\caption{Visual representation of the \textit{Conservative POD-NN} approach, \Cref{sub: PODNN}. In red, the potential space representation (here coinciding with the POD latent space).}
	\label{fig:podnn}
\end{figure}

\subsection{Conservative POD-NN}
\label{sub: PODNN}
Our first proposal is to set $S_{0}:=V_{n}$, i.e. as the POD projector from \Cref{sub: POD}. The POD matrix is constructed starting from the homogeneous part of the flux, thus relying on the snapshots $\{\bq_{0}^{\mu_{i}}\}_{i=1}^{N_{s}}$. This choice is particularly suited if the homogeneous part of the solution manifold, namely
\begin{align}
    \mathscr{S}_{0}:=\{(I-S_{I}B)\bq^{\mu}\;|\;\mu\in M\}\subset\bQ_{0},
\end{align}
exhibits a fast decay of the Kolmogorov $n$-width (equivalently: it can be easily approximated by linear subspaces of reasonably small dimension). In practice, this can be deduced by looking at the decay of the singular values of the snapshot matrix.

In this case, the potential network, $\cN:M\to\mathbb{R}^{n}$, can be implemented as any classical deep feedforward network, and no particular training strategy is required. Recall that $n$ is the number of modes in the reduced basis obtained by POD. Additionally, thanks to the orthonormality of $\mathsf{V_n}$, the minimization of the loss function in \eqref{eq:loss1} can be equivalently replaced with that of
\begin{align}
    \label{eq:loss_POD}
    \loss_{\text{POD}}(\cN)=\frac{1}{N_{s}}\sum_{i=1}^{N_{s}}\|V_{n}^{*}\bq_{0}^{\mu_{i}}-\cN(\mu_{i})\|_{\mathbb{R}^{n}}^{2},
\end{align}
which is computationally more efficient to deal with. This fact is briefly summarized below.

\begin{lemma}
    Let $\mathscr{N}$ be any class of neural network architectures from $M$ to $\mathbb{R}^{n}$. Then,
    \begin{align}
        \argmin_{\cN\in\mathscr{N}}\;\loss(\cN)=\argmin_{\cN\in\mathscr{N}}\;\loss_{\text{POD}}(\cN).
    \end{align}
\end{lemma}
\begin{proof}
   Due to the orthogonality of projection residuals, one has
\begin{align*}
    \|\bq_{0}^{\mu_{i}}-V_{n}\cN(\mu_{i})\|_{\bQ}^{2}
    &=
    \|\bq_{0}^{\mu_{i}}-V_{n}V_{n}^{*}\bq_{0}^{\mu_{i}}\|_{\bQ}^{2}+\|V_{n}V_{n}^{*}\bq_{0}^{\mu_{i}}-V_{n}\cN(\mu_{i})\|_{\bQ}^{2}\\
    &=
    \|\bq_{0}^{\mu_{i}}-V_{n}V_{n}^{*}\bq_{0}^{\mu_{i}}\|_{\bQ}^{2}+\|V_{n}^{*}\bq_{0}^{\mu_{i}}-\cN(\mu_{i})\|_{\mathbb{R}^{n}}^{2}.
\end{align*}
Since the first term in the last equality is independent of $\cN$, the conclusion follows.
\qed\end{proof}

We term this approach ``Conservative POD-NN'' because it forms a natural adaptation of the POD-NN approach introduced in \cite{hesthaven2018non}. The overall idea is illustrated in \Cref{fig:podnn} and summarized in Algorithm \ref{algo:PODNN}.

\newcommand{\ntrain}{N_{s}}
\begin{algorithm}[tb!]
\SetAlgoLined
\SetKwInOut{Input}{Input}
\SetKwInOut{Output}{Output}
\vspace{0.25em}
\Input{FOM solver $\texttt{FOM}=\texttt{FOM}(\mu)$, parameter space $M$, spanning-tree solver $S_{I}$, constraint operator $B$, source function $\texttt{source}$, basis dimension $n$, mass matrix $\mathsf{M}$, sample size $\ntrain$, DNN architecture class $\mathscr{N}$.
\vspace{0.25em}
}

\Output{Trained Conservative POD-NN model.
\vspace{0.75em}
}

$\mathsf{P},\;\mathsf{Q}^{0}=[\;], [\;]$\vspace{0.25em}

\For{\textnormal{$i = 1,\dots,\ntrain$}}{\vspace{0.25em}
$\mu\leftarrow$ random sample from $M$\vspace{0.25em}\\
$\bq\leftarrow \texttt{FOM}(\mu)$\vspace{0.25em}\\
$\bq_{0}\leftarrow (I-S_{I}B)\bq$\vspace{0.25em}\\
$\mathsf{P}\leftarrow \texttt{vstack}(\mathsf{P},\mu)$\vspace{0.25em}\\
$\mathsf{Q}^{0}\leftarrow\texttt{vstack}(\mathsf{Q_{0}},\bq_{0})$\vspace{0.25em}
}\vspace{0.25em}

$\mathsf{V}\leftarrow\texttt{POD}(\mathsf{Q}^{0},\mathsf{M},n)$\vspace{0.25em}\\
$\mathsf{C}\leftarrow\mathsf{Q}^{0}\mathsf{V}$\vspace{0.25em}\\

$\mathcal{N}_{*}\leftarrow\argmin_{\mathcal{N}\in\mathscr{N}}\;\frac{1}{\ntrain}\sum_{i=1}^{\ntrain}\|\mathsf{C}_{i,\cdot}-\mathcal{N}(\mathsf{P}_{i,\cdot})\|^{2}$\vspace{1em}\\

\textbf{def} $\texttt{ROM}=\texttt{ROM}(\mu)$:\\
\hspace{0.5cm} $f\leftarrow\texttt{source}(\mu)$\vspace{0.25em}\\
\hspace{0.5cm} $\bq_{f}\leftarrow S_{I}f$\vspace{0.25em}\\
\hspace{0.5cm} $\bq_{0}\leftarrow \mathsf{V}\mathcal{N}(\mu)$\vspace{0.25em}\\
\hspace{0.5cm}\textbf{return} $\bq_{0}+\bq_{f}$

\vspace{1em}
\Return{\emph{\texttt{ROM}}}\vspace{0.25em}

 \caption{\label{algo:PODNN}Training and implementation of the Conservative POD-NN, Section \ref{sub: PODNN}). Here, \texttt{vstack} denotes the vertical stacking of multiple arrays. \vspace{0.25em}}
 
\end{algorithm}

\subsection{Conservative DL-ROM: Curl-variation}
\label{sub: curlDLROM}
Our second proposal uses the construction in \Cref{sub: diff complex}, thus assuming the presence of a suitable differential complex underlying our problem of interest. For simplicity, we directly focus on the case of the model problem \eqref{eq: model problem} in three dimensions, in which $B$ is the divergence on $\bQ=\mathbb{RT}_0(\Omega_h)$.
We then let 
\begin{align}
    S_{0} &:= \curl, &
    \cN&:M\to\mathbb{N}_0(\Omega_h).
\end{align}

In this case the output of the neural network model is the N\'ed\'elec finite element space, which is high-dimensional. We therefore opt for the DL-ROM training strategy by splitting $\cN=\Psi\circ\phi$ and introducing the auxiliary network $\Psi'$. A visual representation is given in \Cref{fig:dlrom} and Algorithm \ref{algo:DLROM} presents a synthetic summary. Compared to the construction from \Cref{sub: PODNN}, this approach may be better suited when the solution manifold $\mathscr{S}_{0}$ exhibits a complicated structure that is not easily captured by linear subspaces. This situation is typically encountered if the PDE features a strong interaction between the space variable, $x$, and the model parameters, $\mu$: see, e.g., the discussion in \cite{franco2023deep} and the examples reported therein.

With little abuse of notation, we shall use the same terminology also for the 
2D-case, in which $\mathbb{N}_{0}(\Omega_{h})$ is replaced by the nodal space $\mathbb{L}_{1}(\Omega_{h})$, and $\curl$ by $\nabla^{\perp}$, cf. the discussion in \Cref{sub: diff complex}.

\begin{figure}[tb]
	\centering
	\includegraphics[width=\textwidth]{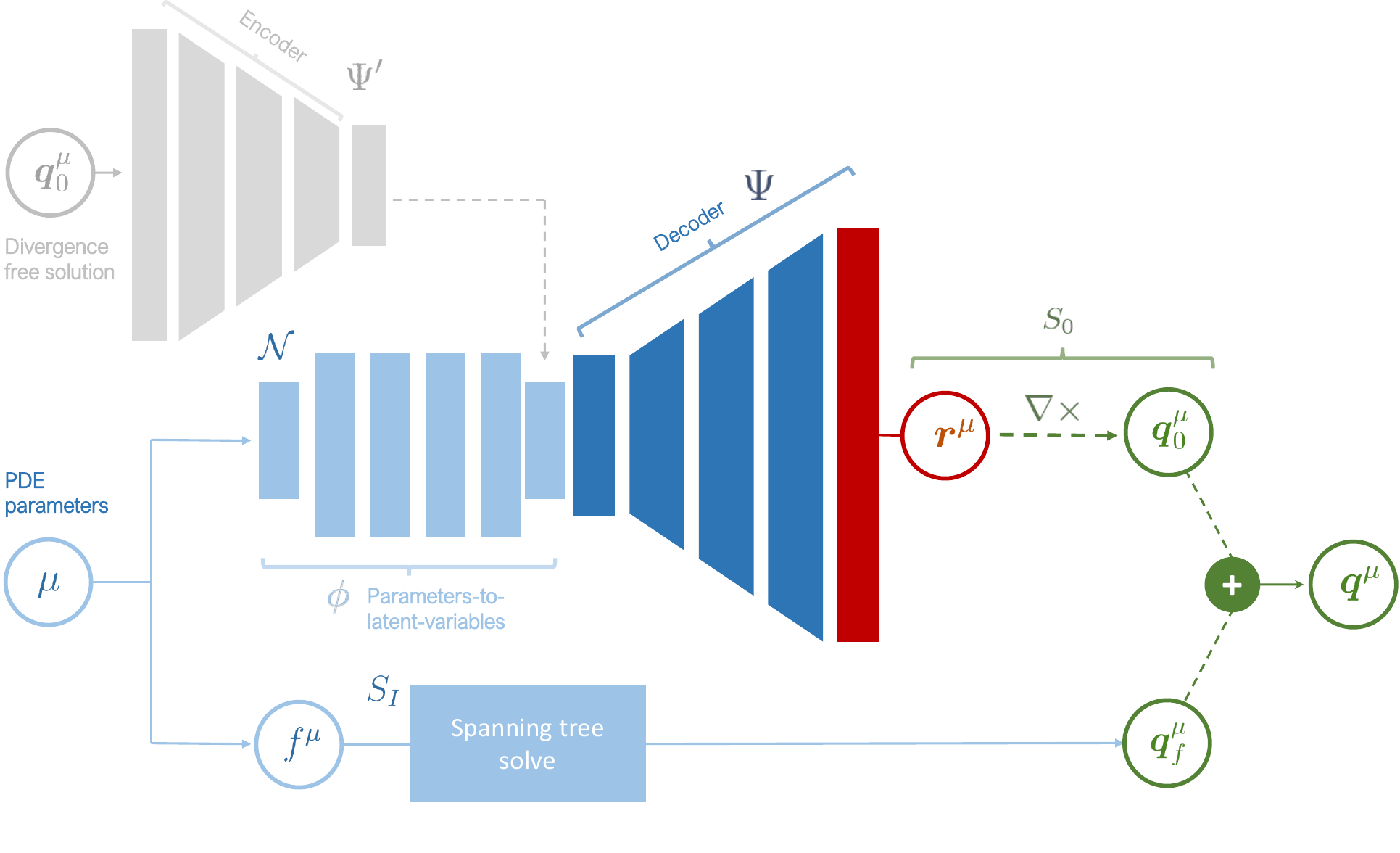}
	\caption{Visual representation of the \textit{Conservative DL-ROM} approach (curl variation), \Cref{sub: curlDLROM}. The potential network, $\cN$ is comprised of two sequential blocks, so that $\cN=\Psi\circ\phi$. An auxiliary encoder network (in gray) is used to introduce an internal regularization: the encoder is discarded after training. In red, the potential space representation.}
	\label{fig:dlrom}
\end{figure}

\begin{algorithm}[tb!]
\SetAlgoLined
\SetKwInOut{Input}{Input}
\SetKwInOut{Output}{Output}
\vspace{0.25em}
\Input{FOM solver $\texttt{FOM}=\texttt{FOM}(\mu)$, parameter space $M$, spanning-tree solver $S_{I}$, constraint operator $B$, source function $\texttt{source}$, mass matrix $\mathsf{M}$, sample size $\ntrain$, type $type$, dictionary of decoder architectures $\mathscr{D}=\mathscr{D}[type]$, class of encoder architectures $\mathscr{E}$, class of latent networks $\mathscr{N}$, regularization hyperparameter $\lambda$, curl matrix $\mathsf{Curl}$ (optional)
\vspace{0.25em}
}

\Output{Trained Conservative DL-ROM model.
\vspace{0.75em}
}

$\mathsf{P},\;\mathsf{Q}^{0}=[\;], [\;]$\vspace{0.25em}

\For{\textnormal{$i = 1,\dots,\ntrain$}}{\vspace{0.25em}
$\mu\leftarrow$ random sample from $M$\vspace{0.25em}\\
$\bq\leftarrow \texttt{FOM}(\mu)$\vspace{0.25em}\\
$\bq_{0}\leftarrow (I-S_{I}B)\bq$\vspace{0.25em}\\
$\mathsf{P}\leftarrow \texttt{vstack}(\mathsf{P},\mu)$\vspace{0.25em}\\
$\mathsf{Q}^{0}\leftarrow\texttt{vstack}(\mathsf{Q^{0}},\bq_{0})$\vspace{0.25em}
}\vspace{0.5em}

\textbf{if}$(type==$"curl"):\\
\hspace{0.5cm}$S_{0}\leftarrow \mathsf{Curl}$\\
\textbf{else}:\\
\hspace{0.5cm}$S_{0}\leftarrow I-S_{I}B$\vspace{1em}\\

$\phi_{*},\Psi_{*},\Psi'_{*}\leftarrow\argmin_{\phi,\Psi,\Psi'\in\mathscr{N}\times\mathscr{D}[type]\times\mathscr{E}}\;\frac{1}{\ntrain}\sum_{i=1}^{\ntrain}[\mathsf{Q}^{0}_{i,\cdot}-S_{0}\Psi(\phi(\mathsf{P}_{i,\cdot}))]^{T}\mathsf{M}[\mathsf{Q}^{0}_{i,\cdot}-S_{0}\Psi(\phi(\mathsf{P}_{i,\cdot}))]$\vspace{0.5em}\\
\hspace{6.5cm}$+\lambda\cdot\frac{1}{\ntrain}\sum_{i=1}^{\ntrain}\|\Psi'(\mathsf{Q}^{0}_{i,\cdot})-\phi(\mathsf{P}_{i,\cdot})\|^{2}$\vspace{1em}\\

$\mathcal{N}\leftarrow\Psi_{*}\circ\phi_{*}$\vspace{1em}\\

\textbf{def} $\texttt{DLROM}=\texttt{DLROM}(\mu)$:\\
\hspace{0.5cm} $f\leftarrow\texttt{source}(\mu)$\vspace{0.25em}\\
\hspace{0.5cm} $\bq_{f}\leftarrow S_{I}f$\vspace{0.25em}\\
\hspace{0.5cm} $\bq_{0}\leftarrow S_{0}\mathcal{N}(\mu)$\vspace{0.25em}\\
\hspace{1cm}\textbf{return} $\bq_{0}+\bq_{f}$

\vspace{1em}
\Return{\emph{\texttt{DLROM}}}\vspace{0.25em}

 \caption{\label{algo:DLROM}Training and implementation of the Conservative DL-ROMs, Sections \ref{sub: curlDLROM}-\ref{sub: sptROM}). \vspace{0.25em}}
 
\end{algorithm}

\subsection{Conservative DL-ROM: Spanning-Tree (SpT) variation}
\label{sub: sptROM}
Our last proposal is to set $S_{0}:=I-S_{I}B$ and, consequently, $\cN:M\to\bQ$. As in the previous case, to handle the high-dimensionality at output, we implement this approach within the DL-ROM framework. Thus, the overall workflow can be represented exactly as in Figure \ref{fig:dlrom}, up to replacing the curl operator with the new choice for $S_{0}$ (see also Algorithm \ref{algo:DLROM}.

Compared to the Conservative POD-NN, this approach has the same advantages as its Curl-counterpart. However, it also comes with increased flexibility as it does not require the existence of an underlying differential complex. It can therefore be readily applied to a larger class of problems (as soon as a suitable right-inverse $S_{I}$ is available).
However, this additional flexibility comes at a cost, which we pay in terms of network complexity. In general, the output of the network, i.e. the potential space $R$, is smaller in the Curl-variation than in the SpT-variation. This fact, whose nature is purely geometrical, is briefly summarized by the two lemmas below.

For 2-dimensional domains, the potential spaces for the curl and the spanning-tree variations are $\mathbb{L}_{1}(\Omega_{h})$ and $\mathbb{RT}_{0}(\Omega_{h})$, respectively. The degrees of freedom of the former correspond to mesh nodes, while those of the latter are represented by the mesh edges. The dimensions of these spaces are related as follows.

\begin{lemma}
    In a 2-dimensional setting, the simplicial tessellation $\Omega_{h}$ satisfies
    \begin{align}
        \label{eq:dim2d}
        \emph{nodes}(\Omega_{h})=\emph{edges}(\Omega_{h})-\left[\emph{elements}(\Omega_{h})-1\right].
    \end{align}
\end{lemma}
\begin{proof}
    This is a reformulation of Euler's formula for planar graphs.
\qed\end{proof}

The difference in output space dimension between the two approaches therefore becomes larger for finer meshes. A similar but weaker result holds in 3D. We recall that, in this case, the potential space of the curl-variation, $\mathbb{N}_{0}(\Omega_{h})$, has the degrees of freedom defined over the mesh edges, while for the SpT-variation, the degrees of freedom are defined over the mesh faces.

\begin{lemma}
    In a 3-dimensional setting, the simplicial tessellation $\Omega_{h}$ satisfies
    \begin{align}
        \label{eq:dim3d}
        \emph{edges}(\Omega_{h})\le\emph{faces}(\Omega_{h})+2.
    \end{align}
\end{lemma}

\begin{proof}
    If $\Omega_{h}$ consists of a single tetrahedron, then $\textnormal{edges}(\Omega_{h})= 6 = \textnormal{faces}(\Omega_{h}) + 2$. Now, in order to address the general case, we note that any expansion of an existing simplicial grid by a new adjacent element (two elements are adjacent iff they share a common face) only increases the difference between the number of edges and faces. There are four possibilities to expand a simplicial grid, depending on how many faces the new element shares with the existing grid: 
    \begin{enumerate}
        \item Adding a new element that shares one face with the grid introduces 3 new faces and 3 new edges, so the difference $\textnormal{faces}(\Omega_{h}) - \textnormal{edges}(\Omega_{h})$ remains the same.
        \item If the new element shares two faces with the grid, then two faces and one new edge need to be added. The difference increases by one.
        \item If the new element shares three faces with the grid, then one only needs to add a single face. The difference increases by one.
        \item The new element shares four faces with the grid, meaning that only the internal volume of the element is missing (no inclusion of extra faces or edges is required).
    \end{enumerate}
    In short, the difference between numbers of faces and edges cannot decrease as the mesh is expanded and the result follows.
\qed\end{proof}

We emphasize that \eqref{eq:dim3d} only provides a weak bound, as one has $\textnormal{edges}(\Omega_{h})<\textnormal{faces}(\Omega_{h})$ in most practical cases. 

\section{Numerical results}
\label{sec: Numerical results}
We assess the performance of the proposed approaches for three different test cases, ranging from mixed-dimensional to non-linear PDEs governing Darcy flow systems.
In the first test case, \Cref{subsec:exp:sines}, we consider a simple 2D problem in which the model parameters present a strong interaction with the space variable. In \Cref{subsec:exp:fractures}, we explore a geometrically more involved scenario, concerning a Darcy flow in a fractured, three-dimensional medium.
Finally, in \Cref{subsec:exp:nonlinear}, we consider a non-linear model based on the Darcy-Forchheimer law.

In order to present a systematic analysis for all three test cases, we follow a standardized protocol that we summarize as follows:
\begin{enumerate}[label = \arabic*)]
    \item We use the FOM to generate a collection of training data, $\{\mu_{i},\bq^{\mu_{i}}\}_{i=1}^{N_{s}}$, which we sample via Latin hypercube over the parameter space. In doing so, we implement the operator $S_{I}$ using the multi-tree variation discussed in Remark \ref{remark: averaged-tree}, specifically using a total of 10 spanning-trees.
    \item We train the three ROMs as detailed in Sections \ref{sub: PODNN}-\ref{sub: sptROM}.
    \item We evaluate the quality of the flux approximations of the three ROMs by computing the relative error below
    \begin{align}
        \mathscr{E}(\text{ROM})=\frac{1}{N_{\text{test}}}\sum_{i=1}^{N_{\text{test}}}\frac{\|\bq_{i}^{\text{test}}-\text{ROM}(\mu_{i}^{\text{test}})\|_{X}}{\|\bq_{i}^{\text{test}}\|_{X},}
    \end{align}
    where $\|\cdot\|_{X}$ is either the $L^{2}$-norm or the $H(\vdiv)$-norm $\| \bq \|_{H(\vdiv)}^2 = \| \bq \|_\Omega^2 + \| \div \bq \|_\Omega^2$, to incorporate the satisfaction of mass conservation. Here, $\{\mu_{i}^{\text{test}},\bq_{_{i}}^{\text{test}}\}_{i=1}^{N_{\text{test}}}$ are an additional collection of FOM data, the so-called test set, that we use for a posteriori evaluations.
    \item We postprocess the flux approximations as to evaluate the quality of the corresponding pressure fields. This is done on the same test set as used in the previous step.
\end{enumerate}

To better justify the effectiveness of the proposed strategies, we also compare the performances of our approaches with those of a naive DNN (Deep Neural Network) regression. That is, we build a benchmark ROM by directly training a neural network model $\Phi$ to learn the map $\mu\to \bq^{\mu}$ in a supervised way, which in practice corresponds to minimizing the loss function below
\begin{align}
    \loss(\Phi):=\frac{1}{N_{s}}\sum_{i=1}^{N_{s}}\|\bq^{\mu_{i}}-\Phi(\mu_{i})\|_{X}^{2},
\end{align}
where, as before, $\|\cdot\|_{X}$ can be either the $L^{2}$ or the $H(\vdiv)$ norm: depending on this choice, we end up with two possible ROMs, which we refer to as ``black-box $L^2$'' and ``black-box $H(\vdiv)$'', respectively. Note that in both cases, the resulting ROM is not guaranteed to yield physical outputs in the sense that it may violate the linear constraint in the PDE model. 

\renewcommand{\thefootnote}{\fnsymbol{footnote}}
Full-order simulations were computed in Python, using PorePy \cite{Keilegavlen2020}, PyGeoN \cite{pygeon} and dedicated modules that are publicly available. Conservative ROMs, instead, were implemented in Pytorch \cite{paszke2019pytorch}. Demonstrative code and trained ROMs are available in our Github repository\footnote[2]{\href{https://github.com/compgeo-mox/conservative_ml}{https://github.com/compgeo-mox/conservative\_ml}}.

\subsection{Dirichlet problem with sinusoidal source}
\label{subsec:exp:sines}

The first test case features problem \eqref{eq: model problem} with $A_{\bq} = I$.
We let $\Omega = (0, 1)^2$ be the domain of interest, which we discretized with a mesh size of $1/32$. This results in a computational grid that has 2400 cells, 3664 faces, and 1265 nodes. 
We set homogeneous boundary conditions for the pressure and a parametrized source function as
\begin{align}
    \label{eq:exp1:source}
    f(x_0, x_1) &= \sin(\mu_0 2 \pi x_0) \sin(\mu_1 2 \pi x_1), &
    \mu_0, \mu_1 &\in [1, 4].
\end{align}
We moreover set a constant vector source term $\bg\equiv[1.0, 0.0]^{T}$ in \eqref{eq:darcy-type}.

We use the FOM solver to sample a total of 2000 different PDE solutions, where $N_{s}=1500$ are used for training and $N_{\text{test}}=500$ for testing. To keep the paper self-contained, we present all the technical details related to the structural hyperparameters of the neural network architectures in \Cref{sec: Appendix - Architectures}.




The results are in Figures \ref{fig:case1flow}-\ref{fig:case1boxplots} and \Cref{tab:exp:sines}. 
We note that, among the black-box approaches, the $L^{2}$-variant performs significantly better than its $H(\vdiv)$-counterpart. This is notable as it highlights that the inclusion of a divergence term in the loss function is not guaranteed to yield better results. Instead, the ROM performance may actually deteriorate. This behavior is likely caused by the intrinsic difficulty in minimizing the loss function, which becomes more and more challenging as we include extra terms in it.

Nevertheless, while the approximations proposed by the $L^{2}$-surrogate are graphically comparable to the reference one, both in terms of flux and pressure, they are not physical, as they violate the conservation of mass. This fact is observed in Figure \ref{fig:case1boxplots} and \Cref{tab:exp:sines}, where we see that the $H(\vdiv)$ relative errors are much higher than those in the $L^{2}$-norm. This indicates a violation of the linear constraint, as for any $\bq,\tilde{\bq}$ satisfying $\div\bq=f=\div\tilde{\bq}$ one would expect that
\begin{equation}
\label{eq:conservative inequality}
\frac{\|\bq-\tilde{\bq}\|_{L^{2}}}{\|\bq\|_{L^{2}}}\ge\frac{\|\bq-\tilde{\bq}\|_{L^{2}}}{\|\bq\|_{H(\vdiv)}}=\frac{\|\bq-\tilde{\bq}\|_{H(\vdiv)}}{\|\bq\|_{H(\vdiv)}},\end{equation}
that is: in relative terms, the errors in $H(\vdiv)$ should be smaller than in $L^{2}$. Indeed, this is what we observe for all conservative approaches. Among these, the best one is undoubtedly the Curl-DL-ROM, which manages to:
i) capture the flux field, with a relative error of 0.1\% and an exact conservation of mass, ii) retrieve a good approximation of the pressure field, and iii) keep the computational cost of the training phase under control. It is noteworthy that the spanning-tree variation of the DL-ROM approach achieves similar results, but at a higher computational cost. On the other hand, the Conservative POD-NN is undeniably the most cost-effective option for training, albeit with slightly inferior performance (see, e.g., the small artifacts and deformations in Figures \ref{fig:case1flow}-\ref{fig:case1pressure}). In this concern, we recall that, due to \eqref{eq:exp1:source}, the problem under study features a strong interaction between the space variable and the model parameters. It therefore is not surprising to see that, even with as many as 100 basis functions (cf. \Cref{sec: Appendix - Architectures}), projection-based techniques do not rank as top performers. 
We also point out that, as clearly indicated in Figure \ref{fig:case1boxplots}, all the conservative approaches perform very well \emph{overall}, i.e., throughout the test set and not merely \emph{on average}. In contrast, the performances of the black-box surrogates are much more volatile, especially in $H(\vdiv)$ norm.


Finally, it is interesting to note that, in line with Lemma \ref{lemma: pressure}, the ROMs reporting the smallest $L^{2}$ errors on the flux are also those providing a better approximation of the pressure fields. While the errors increase by two orders of magnitude, we consider these results satisfactory, as the pressure field was never explicitly involved during the training phase. We emphasize that the main purpose of this work is to provide a reliable approximation of the flux component. Further improvements are required if a higher quality of the pressure approximation is required, but that is beyond the scope of this work.

\definecolor{Gray}{gray}{0.95}
\renewcommand{\arraystretch}{1.5}
\begin{table}
    \centering
    \caption{Models comparison for the first case study, \Cref{subsec:exp:sines}. All the errors reported refer to the average performance over the test set.}
    \begin{tabular}{lllllr}
    \hline
     \textbf{Model} & \textbf{Map onto}\;\;& \textbf{$L^{2}$ error} & \textbf{$H(\vdiv)$ error} & \textbf{$L^{2}$ error}&\textbf{Training time}\vspace{-0.4em}\\
     &\textbf{kernel} $S_{0}$\;\;&Flow $\bq$&Flow $\bq$&Pressure $p$&\\
     \hline\hline
     Black-box $L^{2}$& None & \;\;\;0.20\% &\;\;\;\;\;\;\;\;\;\;9.00\% & \;\;\;\;10.52\%& 4m 50.0s\\
     Black-box $H(\vdiv)$& None & \;\;\;1.79\% &\;\;\;\;\;\;\;\;13.35\% & \;\;146.30\% & 6m 13.4s\\
     Conserv. POD-NN & $V_n$ & \;\;\;0.17\% &\;\;\;\;\;\;\;\;\;\;0.15\% & \;\;\;\;40.54\%& 2m 50.1s\\
     \rowcolor{Gray} Conserv. DL-ROM & $\curl$ & \;\;\;0.10\%&\;\;\;\;\;\;\;\;\;\;0.09\% &\hspace{1.5em}9.31\%& 4m 52.7s\\
      Conserv. DL-ROM & $I-S_{I}B$ & \;\;\;0.15\% &\;\;\;\;\;\;\;\;\;\;0.14\% &\hspace{1.0em}32.50\% & 8m 5.02s\\\hline\\
    \end{tabular}
    \label{tab:exp:sines}
\end{table}
\begin{figure}
    \centering
    \includegraphics[width=0.3\textwidth]{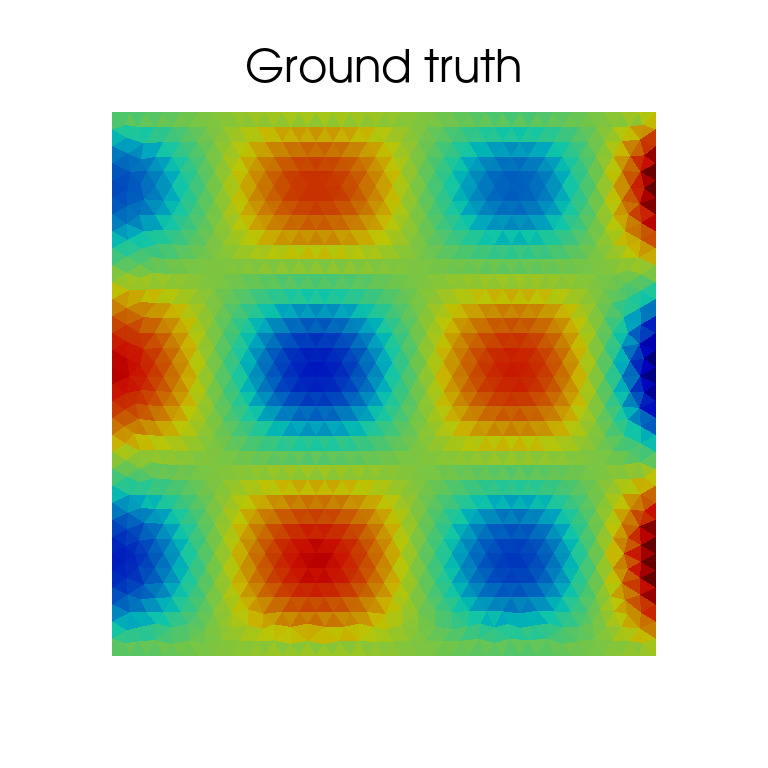}
    \includegraphics[width=0.3\textwidth]{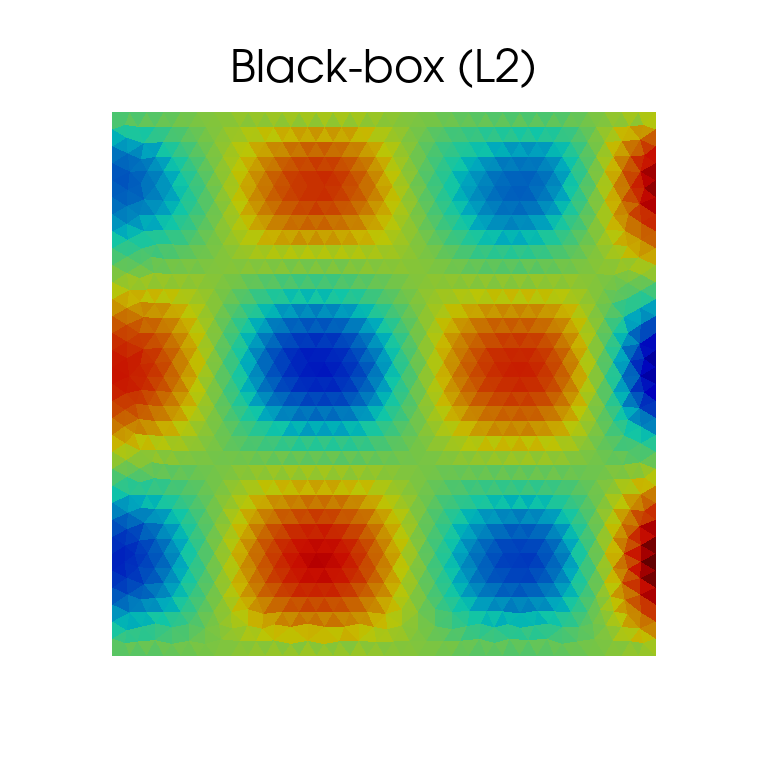}
    \includegraphics[width=0.3\textwidth]{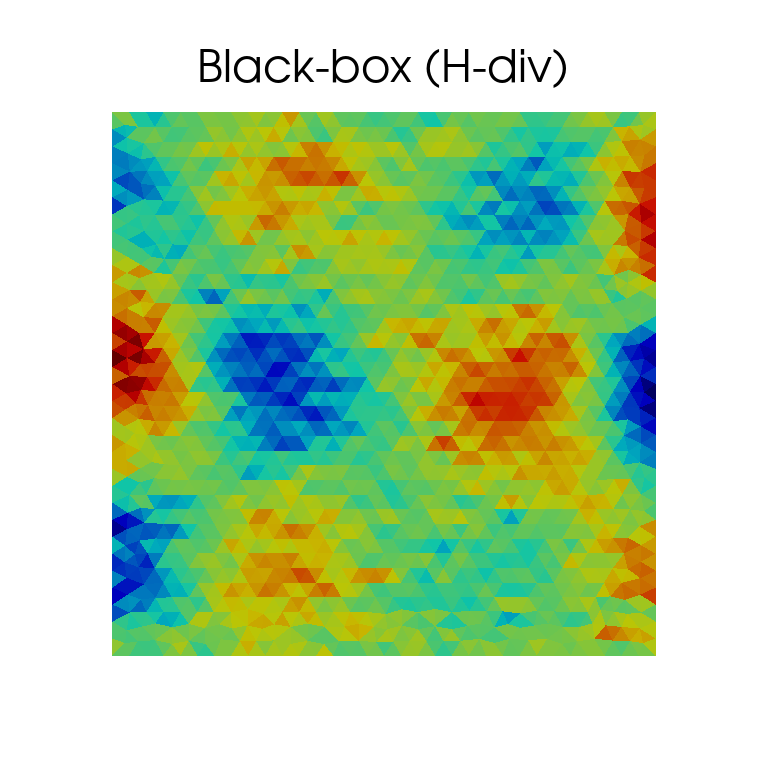}
    \includegraphics[height=3.45cm]{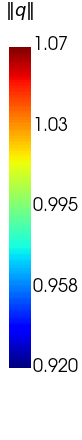}
    \\\vspace{-0.5cm}
    \includegraphics[width=0.3\textwidth]{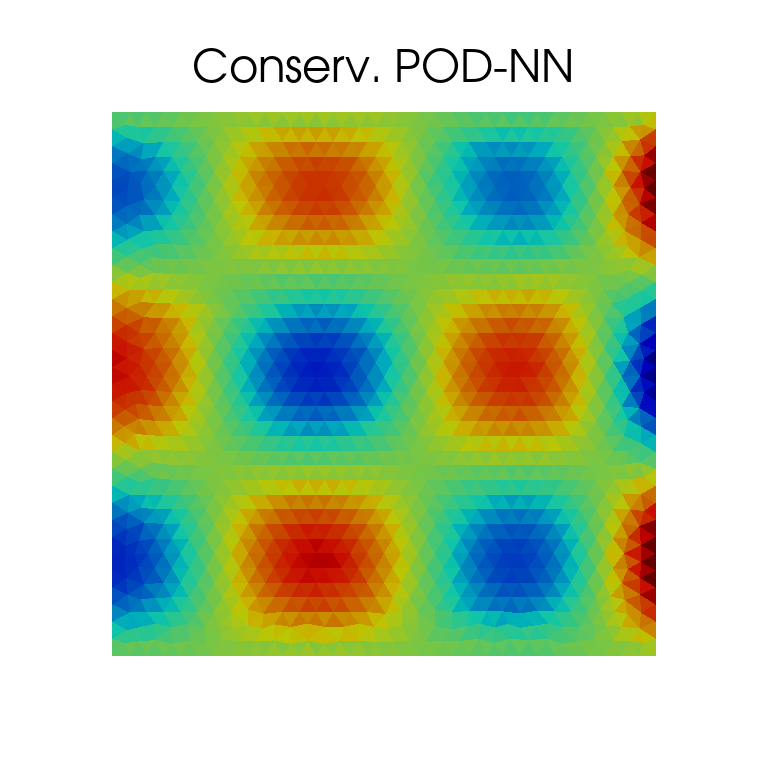}
    \includegraphics[width=0.3\textwidth]{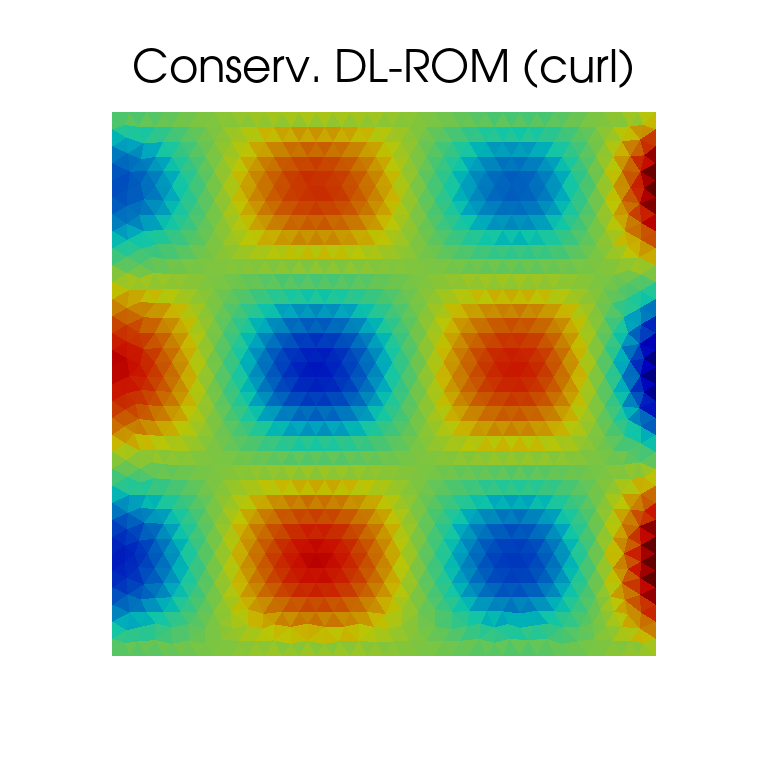}
    \includegraphics[width=0.3\textwidth]{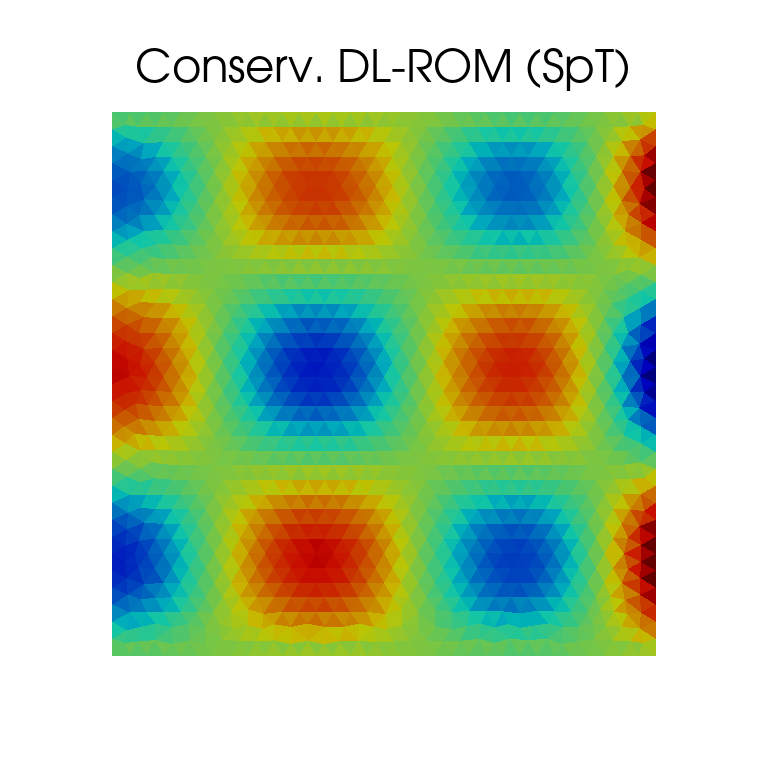}
    \includegraphics[height=3.45cm]{fig/case1/case1flow-bar.png}
    \caption{Magnitude of the velocity fields $|\bq|$ for Case 1, \Cref{subsec:exp:sines}. FOM solution (top left) vs ROM approximations for an unseen configuration of the problem parameters, $\bmu=[1.332, 1.423].$ See Figure \ref{fig:case1pressure} for the corresponding pressure fields.}
    \label{fig:case1flow}
\end{figure}
\begin{figure}
    \centering
    \includegraphics[width=0.3\textwidth]{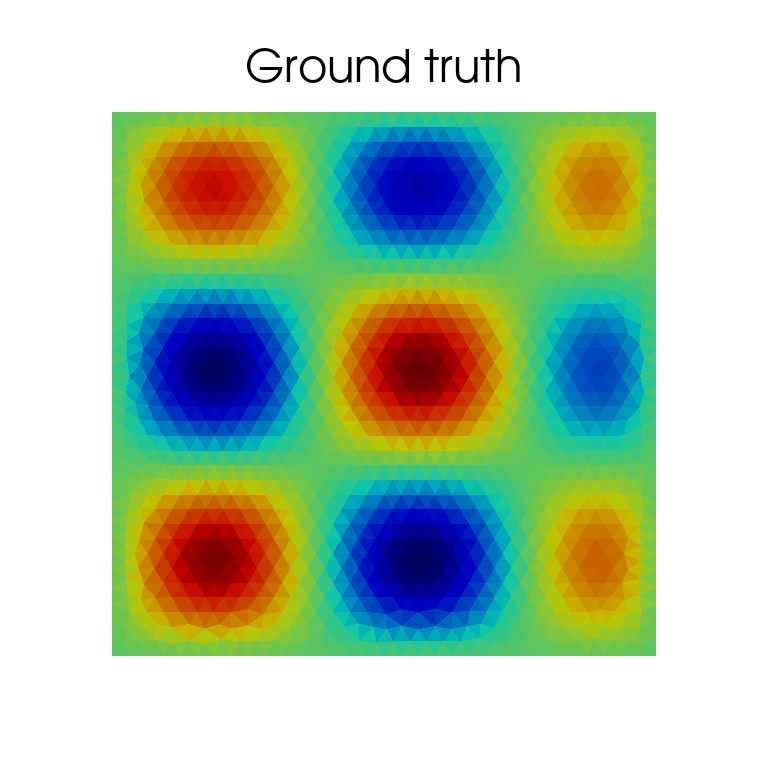}
    \includegraphics[width=0.3\textwidth]{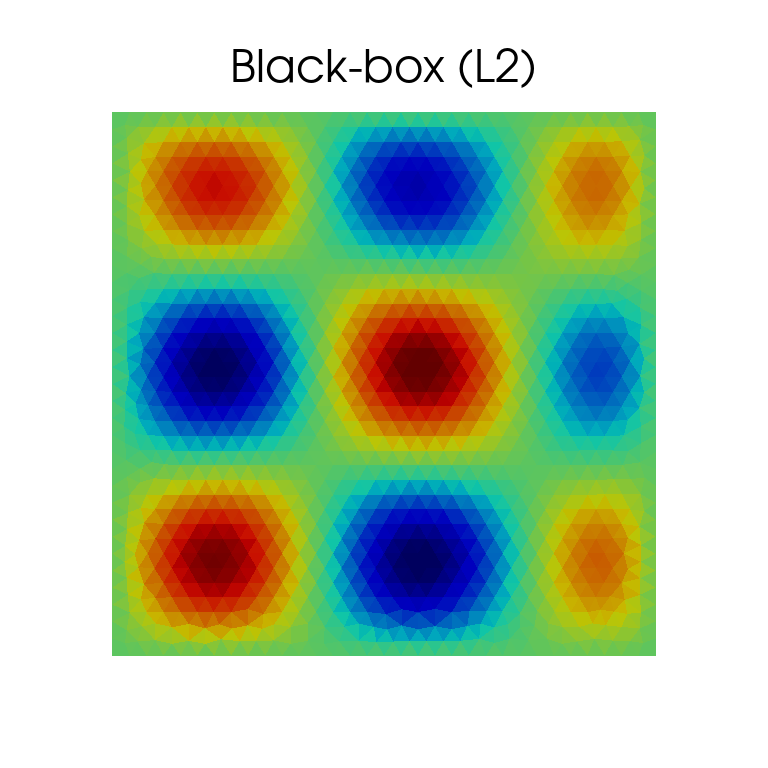}
    \includegraphics[width=0.3\textwidth]{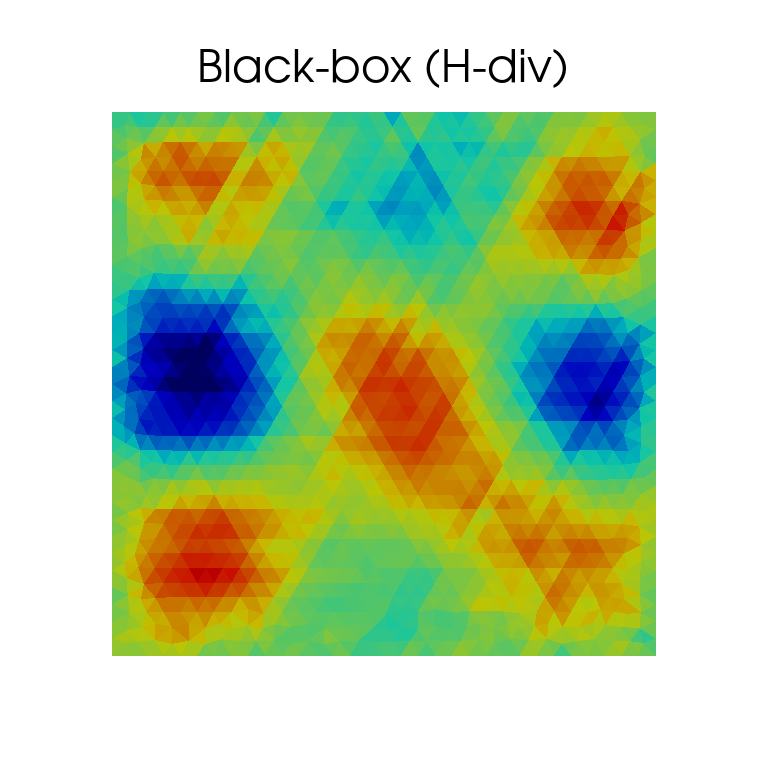}
    \includegraphics[height=3.45cm]{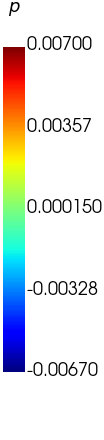}\\
    \includegraphics[width=0.3\textwidth]{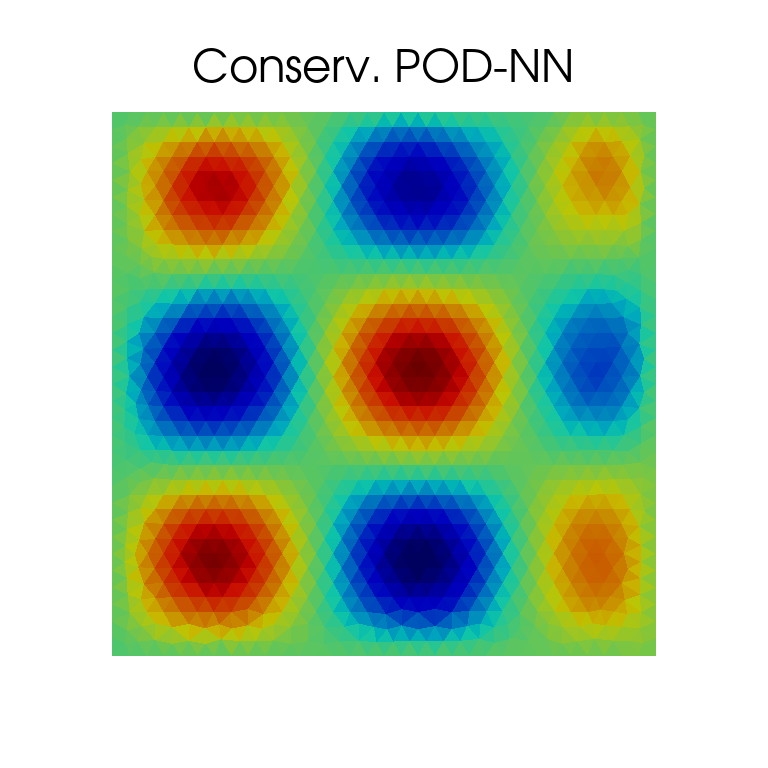}
    \includegraphics[width=0.3\textwidth]{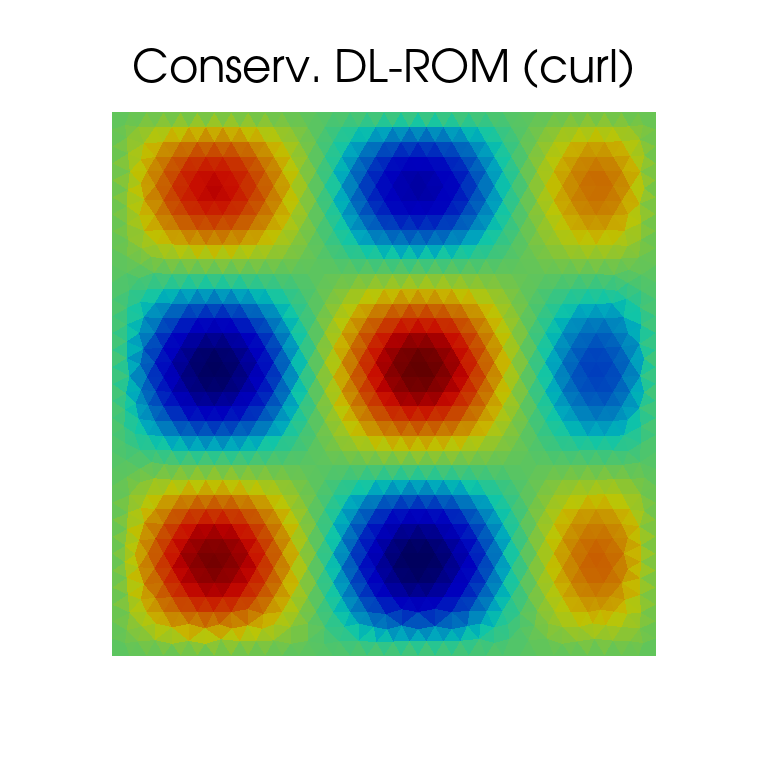}
    \includegraphics[width=0.3\textwidth]{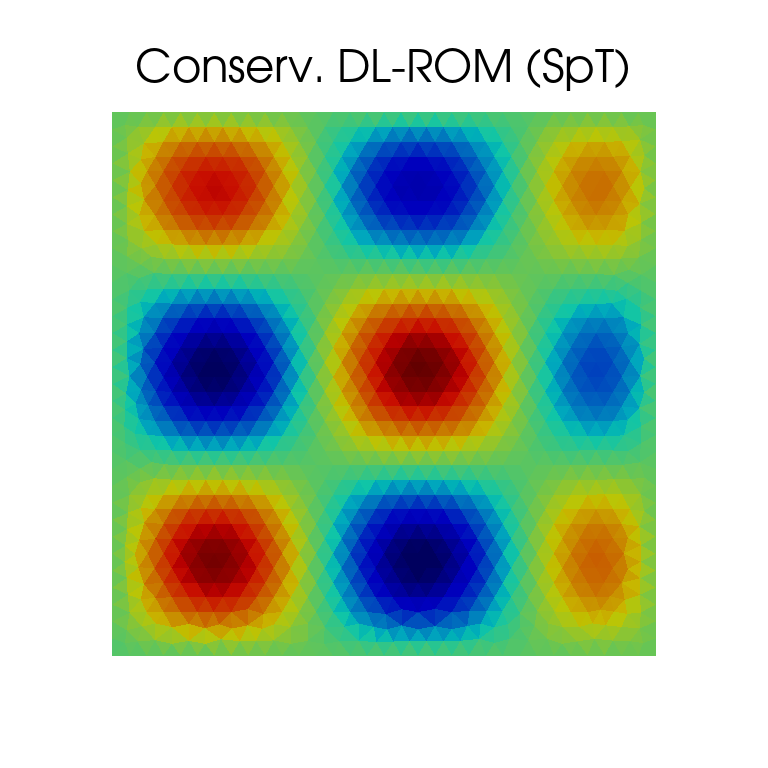}
    \includegraphics[height=3.45cm]{fig/case1/case1pressure-bar.png}
    \caption{Pressure fields $p$ for Case 1, \Cref{subsec:exp:sines}. FOM solution (top left) vs ROM approximations for an unseen configuration of the problem parameters, $\bmu=[1.332, 1.423].$ See Figure \ref{fig:case1flow} for the corresponding flow fields.}
    \label{fig:case1pressure}
\end{figure}



\begin{figure}
    \centering
    \includegraphics[width=\textwidth]{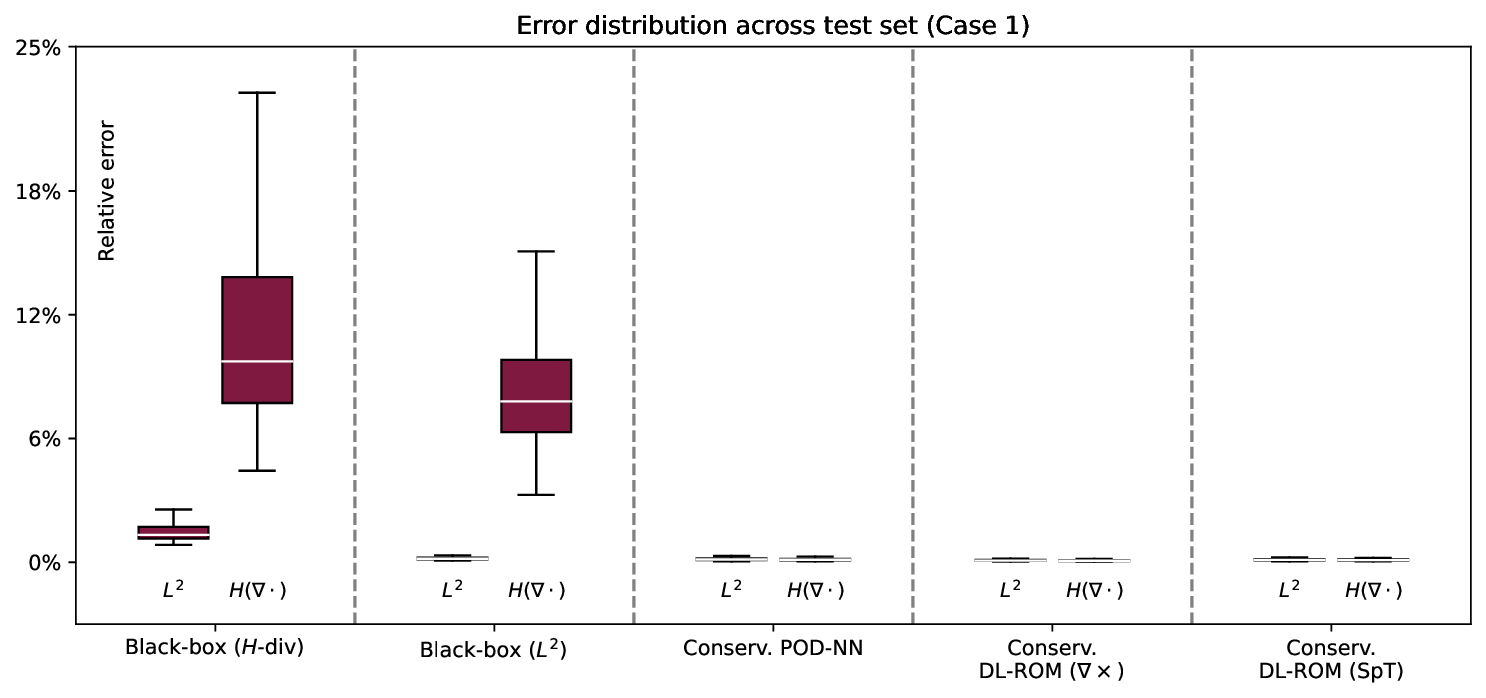}
    \caption{Flux error distributions for the first case study, \Cref{subsec:exp:sines}.}
    \label{fig:case1boxplots}
\end{figure}

\subsection{Darcy flow in a fractured porous medium}
\label{subsec:exp:fractures}

As a second test case, we consider flow in a fractured porous medium in which five fractures
are represented by planar two-dimensional inclusions in a three-dimensional medium. 
We provide a concise presentation of the model in \Cref{sec: Mixed-dimensional model for flow in fractured porous media}. 

We set the conductivity of the bulk matrix to unity and set the aperture constant as $\varepsilon_i = 10^{-4}$ for all subdomains $\Omega_i$. Let us consider the following parametrization for the fracture permeability $K_{frac}$, the pressure boundary conditions, and source term $f$:
\begin{subequations}
\begin{align}
    \label{eq: fractures}
    K_{frac} &= 10^{\mu_0}, &
    \mu_0 &\in [3, 5], \\
    p|_{\partial \Omega}(\bm{x}) &= [\mu_1, \mu_2, \mu_3]^\top \cdot \bm{x}, &
    \mu_1, \mu_2, \mu_3  &\in [0, 1], \\
    f(\bm{x}) &=
    \begin{cases}
        \mu_4, & \bm{x} \in \Omega_{frac} \\
        0, & \bm{x} \in \Omega_{bulk}
    \end{cases} &
    \mu_4 &\in [1, 2].
\end{align}
\end{subequations}
For the FOM solver, we rely on a mixed-dimensional mesh with mesh size $h=0.1$, which results in a total of $17249$ degrees for the flux space. The FOM is used to generate a total of 1000 snapshots, $N_{s}=800$ for training and $N_{\text{test}}=200$ for testing.

We report in Figures \ref{fig:case2flow}-\ref{fig:case2pressure} the solutions obtained with the FOM along with some of the reduced approaches. Since a proper visualization requires plotting both the bulk and the fractures, we have opted to show only the best and worst performing methods of the five ROMs.

As for the previous test case, we see that the approximations provided by the black-box surrogates are visibly unphysical. Furthermore, the results confirm that a forceful inclusion of the divergence term in the loss function is not beneficial, making the $H(\vdiv)$ approach the worst performer. Instead, all conservative ROMs are capable of providing a good approximation of the flux $\bq$, with $L^{2}$ and $H(\vdiv)$ errors below 3\%. 

This time, among the three, POD-NN excels, with average relative errors of 1.16\% and 0.50\% for the flux and pressure, respectively. Furthermore, it is the fastest ROM to train, as it only requires the optimization of a small DNN architecture (here, the potential space was $\dim(\bR)=\dim(V_{n})=n=10)$. Once again, this is not particularly surprising as the underlying PDE is mostly diffusive and the parameters do not directly interact with the space variable. Thus, the solution manifold associated with \eqref{eq: fractures} is expected to present a fast decay of its Kolmogorov $n$-width, meaning that it can be easily approximated by linear subspaces (such as the one given by the POD basis). This means that, while all conservative ROMs are satisfactory, prior knowledge of the problem can guide in the selection of the appropriate ROM.

\begin{figure}
    \centering
    \includegraphics[width=0.29\textwidth]{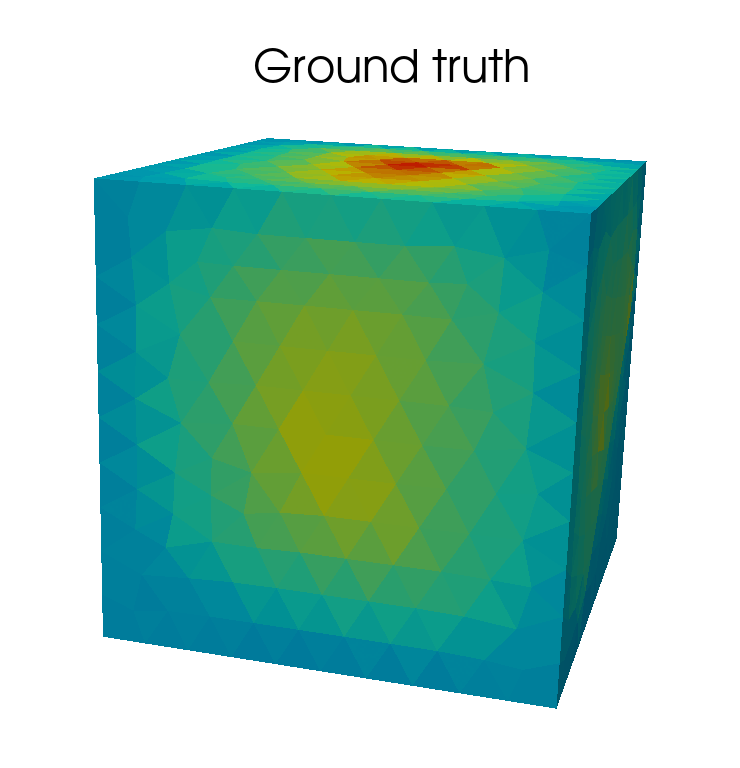}
    \includegraphics[width=0.29\textwidth]{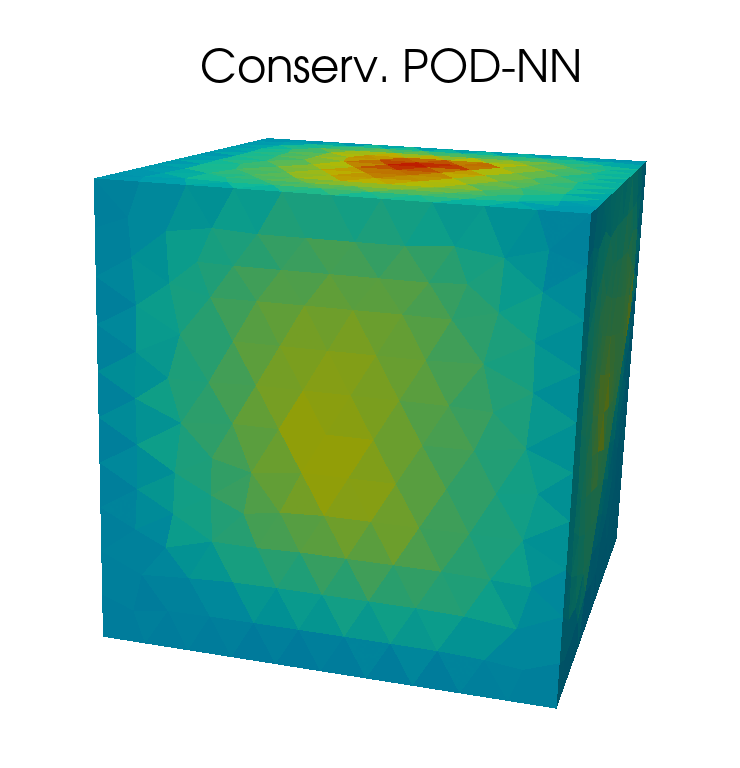}
    \includegraphics[width=0.29\textwidth]{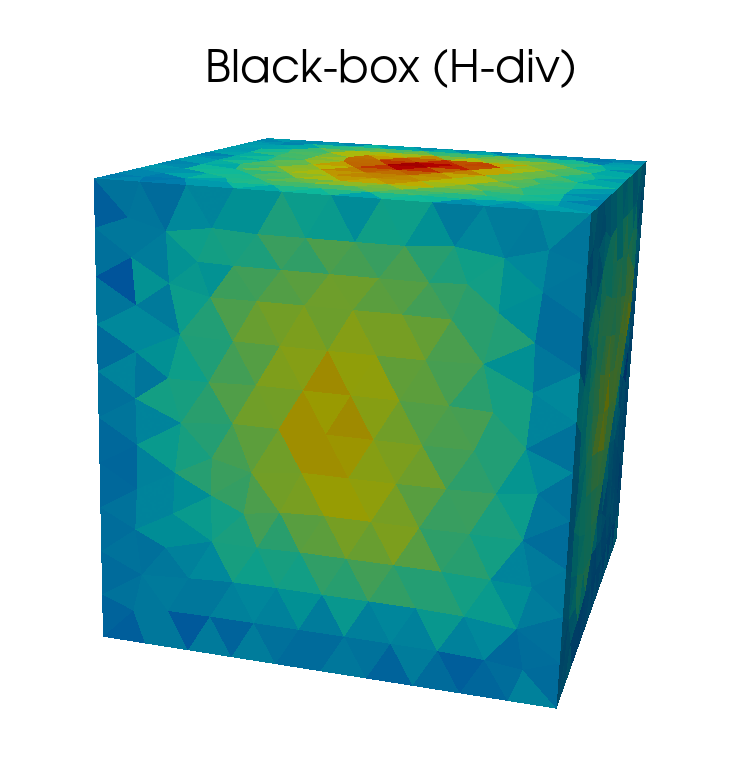}
    \includegraphics[height=3.45cm]{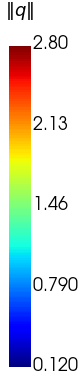}\\
    \includegraphics[width=0.29\textwidth]{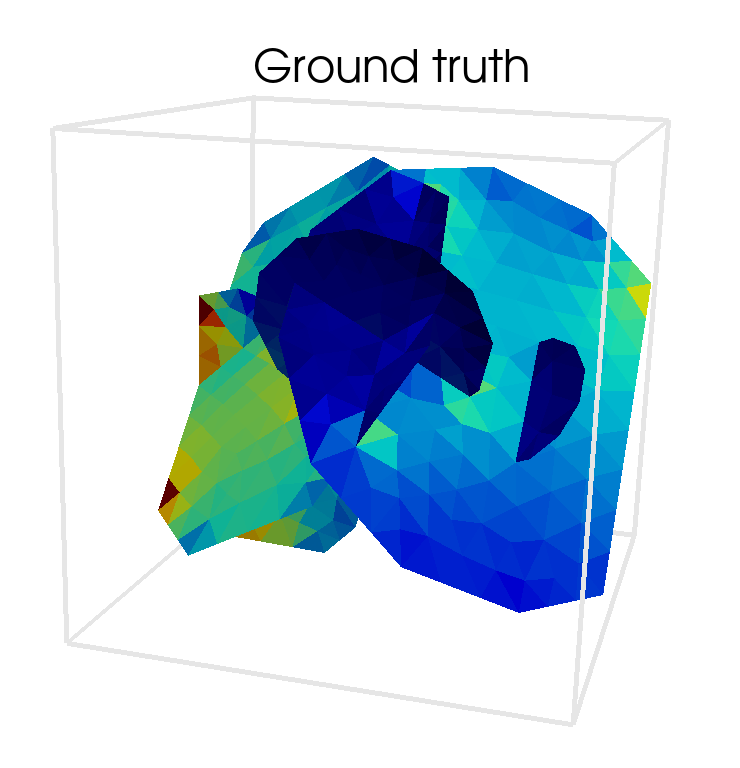}
    \includegraphics[width=0.29\textwidth]{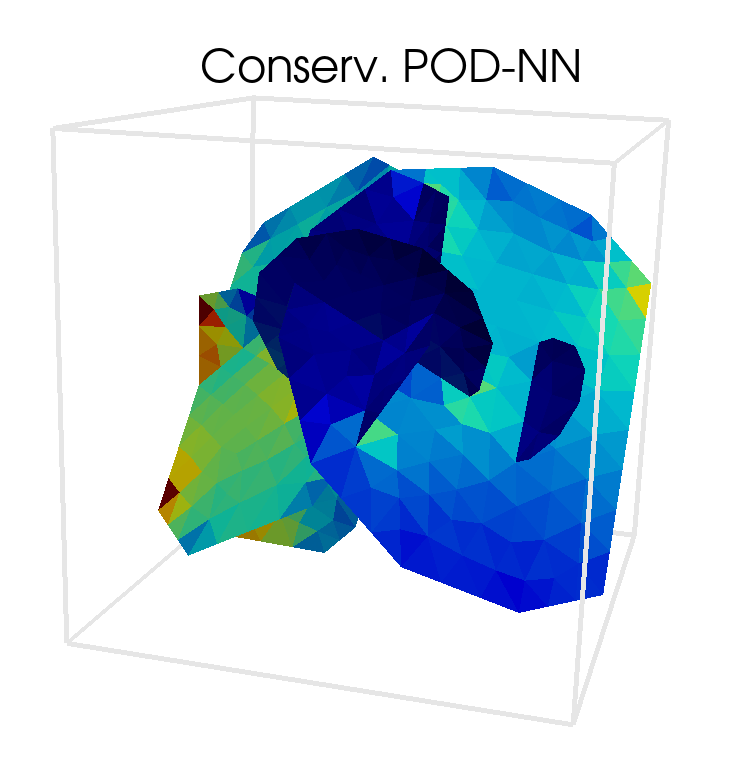}
    \includegraphics[width=0.29\textwidth]{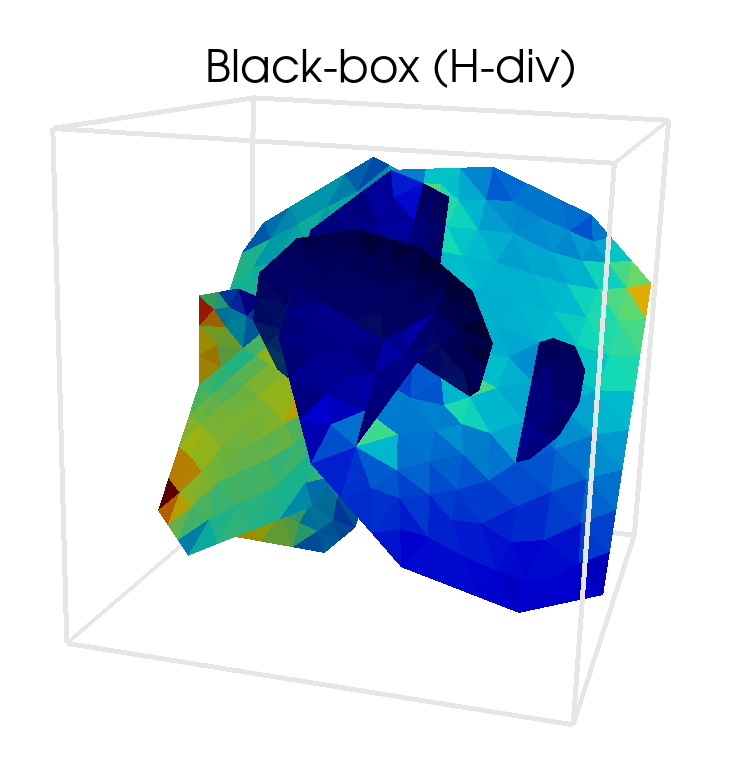}\vspace{0.5em}
    \includegraphics[height=3.45cm]{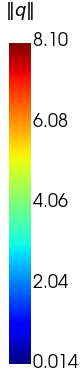}
    \caption{Magnitude of the velocity fields $|\bq|$ for Case 2, \Cref{subsec:exp:fractures}. FOM solution (left) vs ROM approximations for an unseen configuration of the problem parameters, $\bmu=[4.947, 0.406, 0.713, 0.552, 1.651].$ The top row depicts the bulk flux whereas the bottom row concerns the fracture flux. The bulk colorbar has been rescaled to match surface values. See Figure \ref{fig:case2pressure} for the corresponding pressure fields.}
    \label{fig:case2flow}
\end{figure}
\begin{figure}
    \centering
    \includegraphics[width=0.29\textwidth]{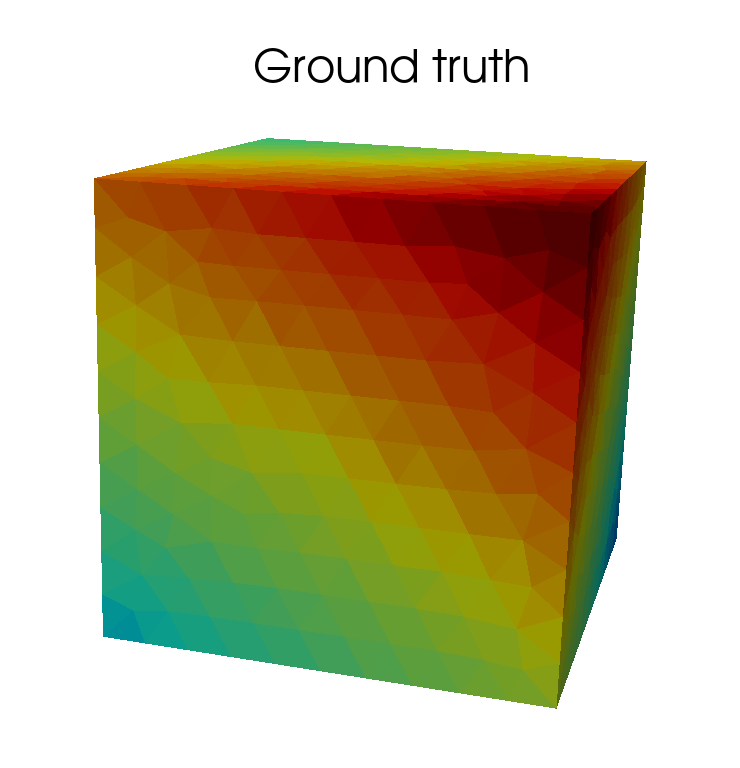}
    \includegraphics[width=0.29\textwidth]{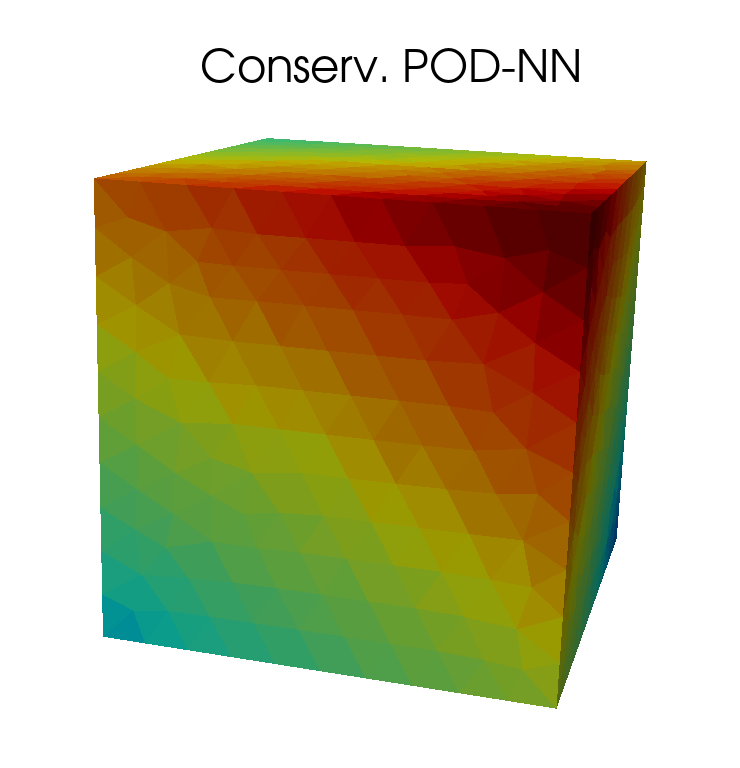}
    \includegraphics[width=0.29\textwidth]{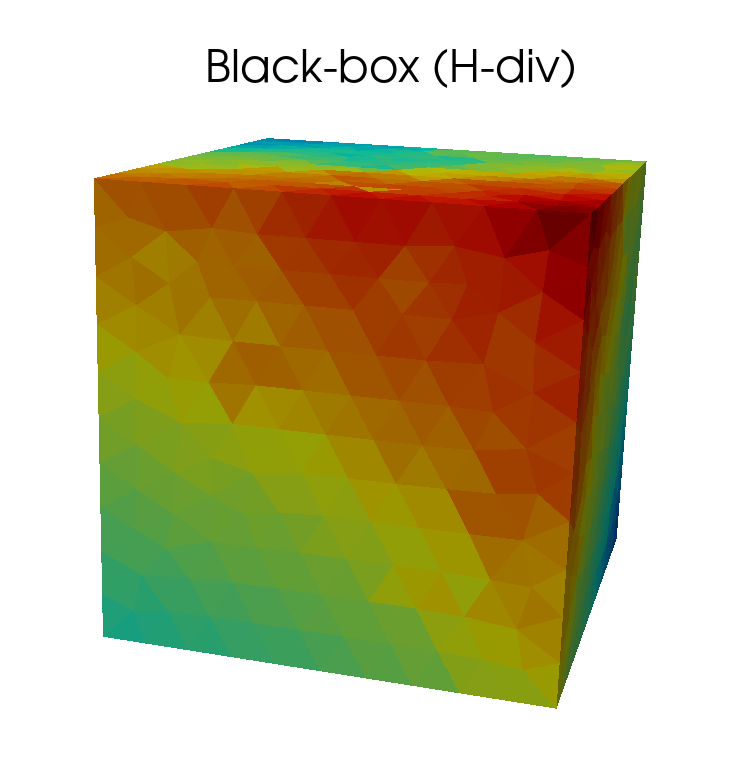}
    \includegraphics[height=3.45cm]{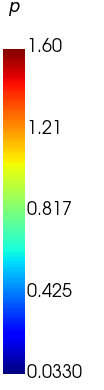}\\
    \includegraphics[width=0.29\textwidth]{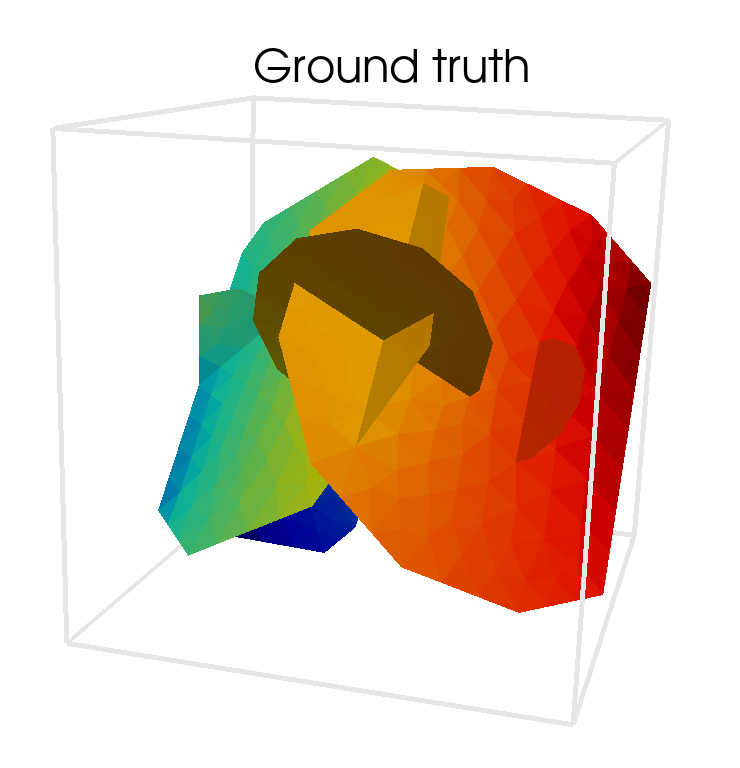}
    \includegraphics[width=0.29\textwidth]{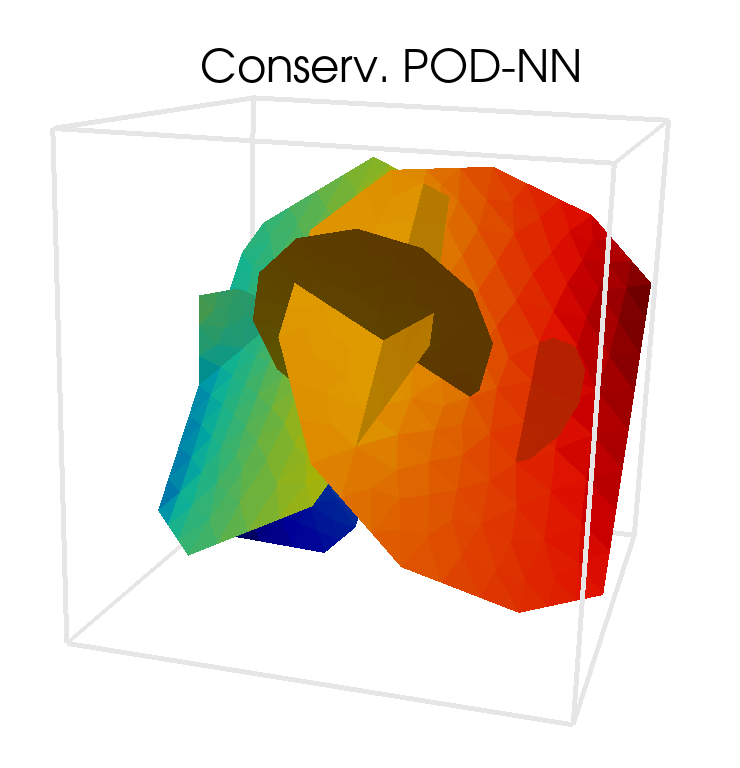}
    \includegraphics[width=0.29\textwidth]{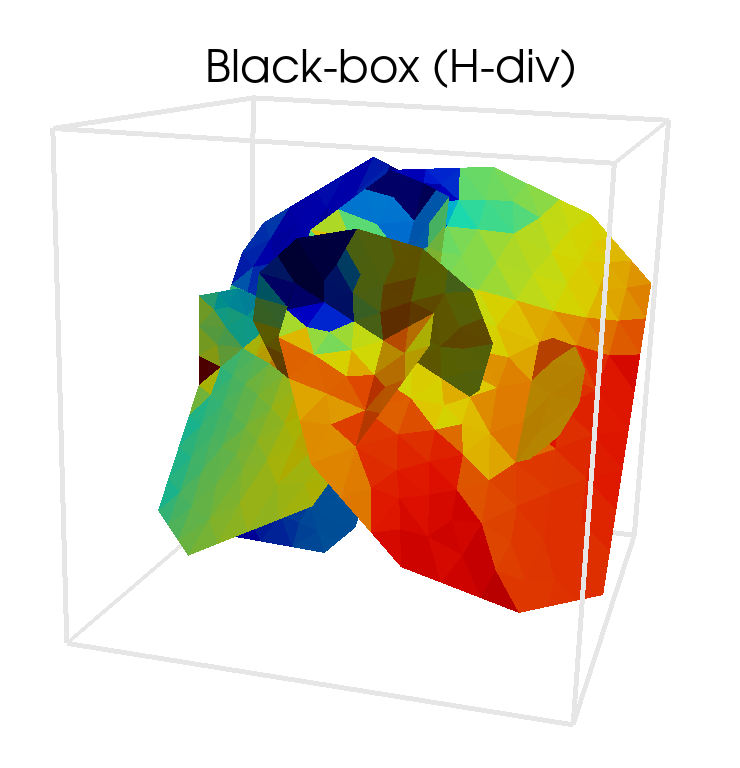}
    \includegraphics[height=3.45cm]{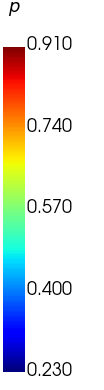}
    \caption{Pressure fields $p$ for Case 2, \Cref{subsec:exp:fractures}. FOM solution (top left) vs ROM approximations for an unseen configuration of the problem parameters, $\bmu=[4.947, 0.406, 0.713, 0.552, 1.651].$ The top row depicts the bulk pressure whereas the bottom row concerns the fracture pressure. See Figure \ref{fig:case2flow} for the corresponding flow fields.}
    \label{fig:case2pressure}
\end{figure}
\begin{table}
    \label{tab:case2errors}
    \centering
    \caption{Models comparison for the second case study, \Cref{subsec:exp:fractures}. Table entries read as in Table \ref{tab:exp:sines}. $\fD\times$ = generalized mixed-dimensional Curl, see \Cref{sec: Mixed-dimensional model for flow in fractured porous media}.}
    \begin{tabular}{lllllr}
    \hline
     \textbf{Model} & \textbf{Map onto}\;\;& \textbf{$L^{2}$ error} & \textbf{$H(\vdiv)$ error} & \textbf{$L^{2}$ error}&\textbf{Training time}\vspace{-0.4em}\\
     &\textbf{kernel} $S_{0}$\;\;&Flow $\bq$&Flow $\bq$&Pressure $p$&\\
     \hline\hline
     Black-box $L^{2}$& None & \;\;\;1.25\% &\;\;\;\;\;\;\;\;\;14.50\% & \;\;\;\;10.15\%& 24m 24.5s\\
     Black-box $H(\vdiv)$& None & \;13.19\% &\;\;\;\;\;\;\;\;\;14.05\% & \;\;317.81\% & 33m 14.7s\\
     \rowcolor{Gray} Conserv. POD-NN & $V_n$ & \;\;\;1.16\% &\;\;\;\;\;\;\;\;\;\;\;1.16\% & \hspace{1.5em}0.50\%& 3m 08.4s\\
     Conserv. DL-ROM & $\fD \times$ & \;\;\;2.59\%&\;\;\;\;\;\;\;\;\;\;\;2.59\%&\hspace{1.em}69.52\%& 31m 25.2s\\
     Conserv. DL-ROM & $I-S_{I}B$ & \;\;\;1.20\% &\;\;\;\;\;\;\;\;\;\;\;1.20\% &\hspace{1.0em}44.47\% & 48m 16.9s\\\hline\\
    \end{tabular}
    \label{tab:exp:fractures}
\end{table}
\begin{figure}
    \centering
    \includegraphics[width=\textwidth]{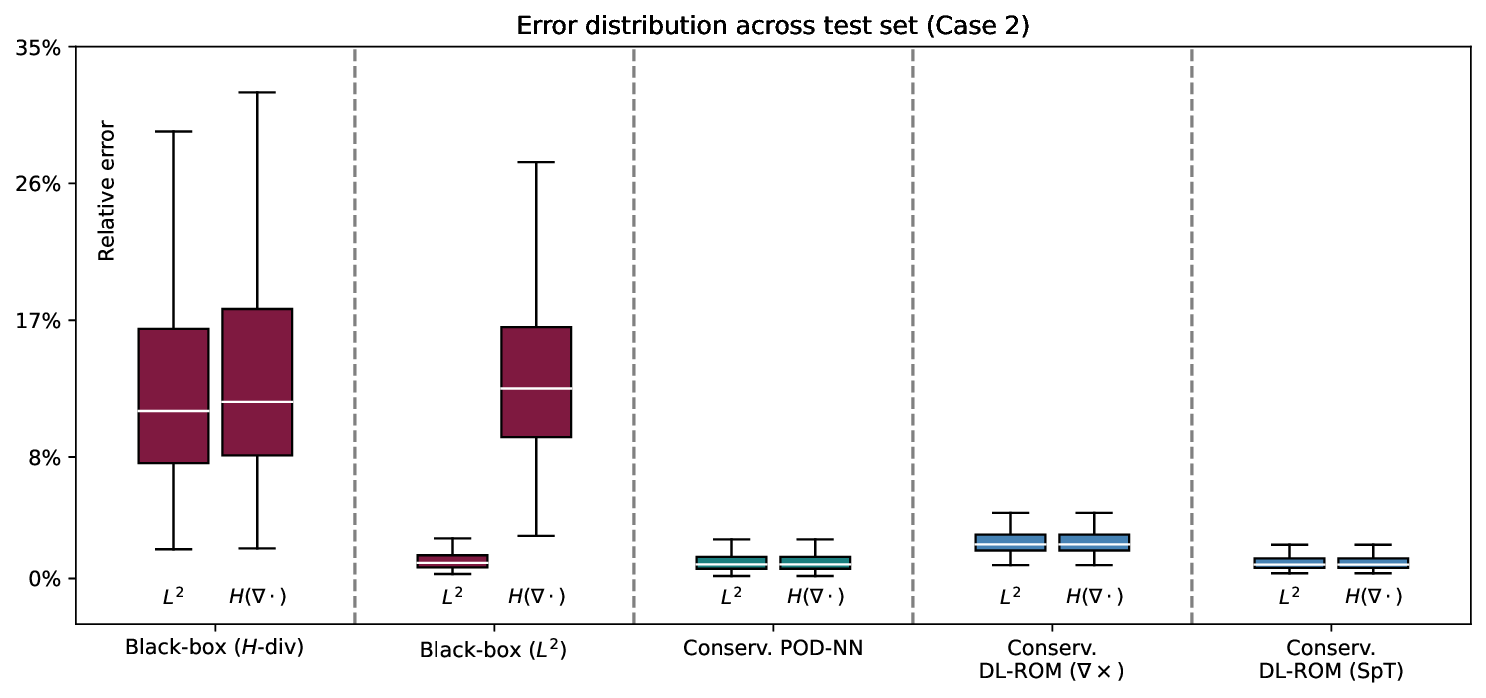}
    \caption{Flux error distributions for the second case study, \Cref{subsec:exp:fractures}.}
    \label{fig:case2boxplots}
\end{figure}

\subsection{Darcy-Forchheimer flow}\label{subsec:exp:nonlinear}

As a final test case, we consider a Darcy-Forchheimer flow system on $\Omega := (0, 1)^2$, given by the following nonlinear system of equations:
\begin{subequations}
\begin{align}
    \kappa_0^{-1}(1 + \kappa_1^{-1} |\bq|) \bq + \nabla p &= 0, &
    \text{in }&\Omega, \\
    \div \bq &= f, &
    \text{in }&\Omega, \\
    p &= g, &
    \text{on }&\partial \Omega.
\end{align}
\end{subequations}

We assume that the source terms, material parameters, and boundary conditions are parametrized as follows,
\begin{subequations}
\begin{align}
    f(x_0, x_1) &= \mu_0 \sin(2 \pi x_0) + (1-\mu_0) \sin(2 \pi x_1), &
    \mu_0 &\in [0, 1], \\
    g(x_0, x_1) &= \mu_1 x_0 x_1, & \mu_1 &\in [0, 1],\\
    \kappa_0 &= 10^{\mu_2}, \
    \kappa_1 = 10^{\mu_3}, &
    \mu_2 &\in [-2, 1], \
    \mu_3 \in [0, 2].
\end{align}
\end{subequations}
For the FOM solver, we rely on a structured triangular grid of size $h = 1/32$. We generate a total of $N_{s}=800$ training snapshots and $N_{\text{test}}=200$ testing instances.

We remark that this problem constitutes a prototypical scenario in which our reduced order modeling approach would be of particular interest because i) it features a fairly expensive FOM (nearly 1m 15s per simulation, too costly for multi-query applications) and ii) it concerns a nonlinear PDE, which poses significant challenges for intrusive ROMs based on linear projections, such as, e.g., POD-Galerkin methods.
As before, all the technical details about the design of ROMs and DNN architectures can be found in \Cref{sec: Appendix - Architectures}.

The results are presented in Figures \ref{fig:case3all}-\ref{fig:case3boxplots} and \Cref{tab:exp:nonlinear}. As for the previous test cases, the $H(\vdiv)$-black-box model ranks as the worst performer, highlighting once again how ``physics-informed'' loss functions can be a double edged sword. 
On the other hand, while the $L^{2}$-surrogate appears to be reasonably accurate, it is still physically inconsistent, as it returns flux fields that violate conservation of mass, cf. \eqref{eq:conservative inequality} and \Cref{tab:exp:nonlinear}. This is also evident in Figure \ref{fig:case3boxplots}, where we clearly see that the $H(\vdiv)$ errors are much higher and much more spread compared to the $L^{2}$ ones.

In contrast, all conservative ROMs show satisfactory results, both in terms of flux and pressure approximation, with Curl-DL-ROM reporting the best performance. This applies not only to the average expressivity (\Cref{tab:exp:sines}), but also to the overall quality of the approximation, as depicted in Figure \ref{fig:case3boxplots}. There, in fact, we see that the relative errors of conservative ROMs are subject to fluctuations, but always fall below 6\%. This is especially true for Conservative POD-NN and Conservative Curl-DL-ROM, whose performance is proven to be rather robust, as indicated by the width of the box plots. For instance, we see that more than 150 out of 200 test instances were approximated with $L^{2}$ errors below $3\%.$

\begin{figure}
    \centering
    \includegraphics[width=0.3\textwidth]{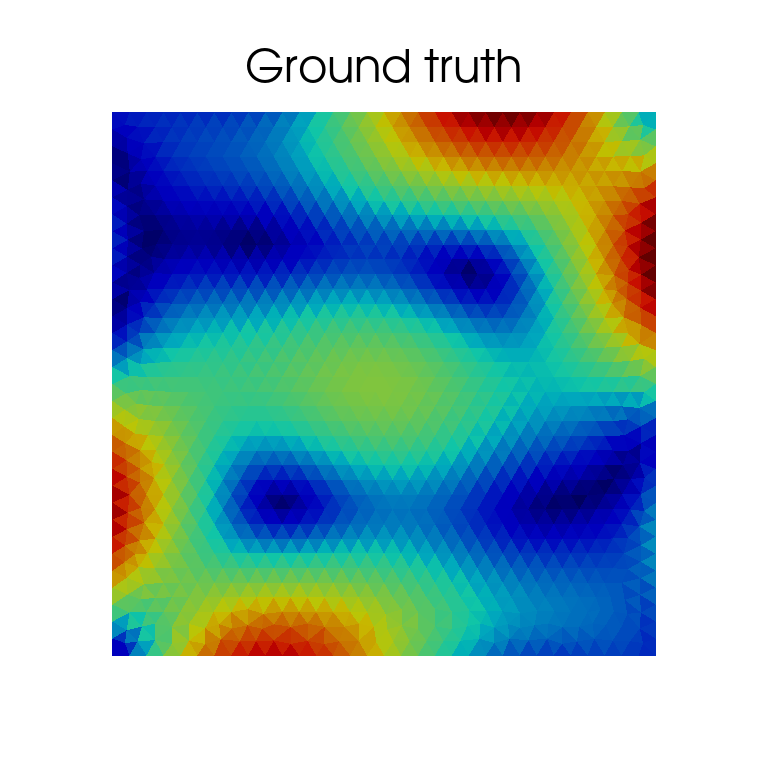}
    \includegraphics[width=0.3\textwidth]{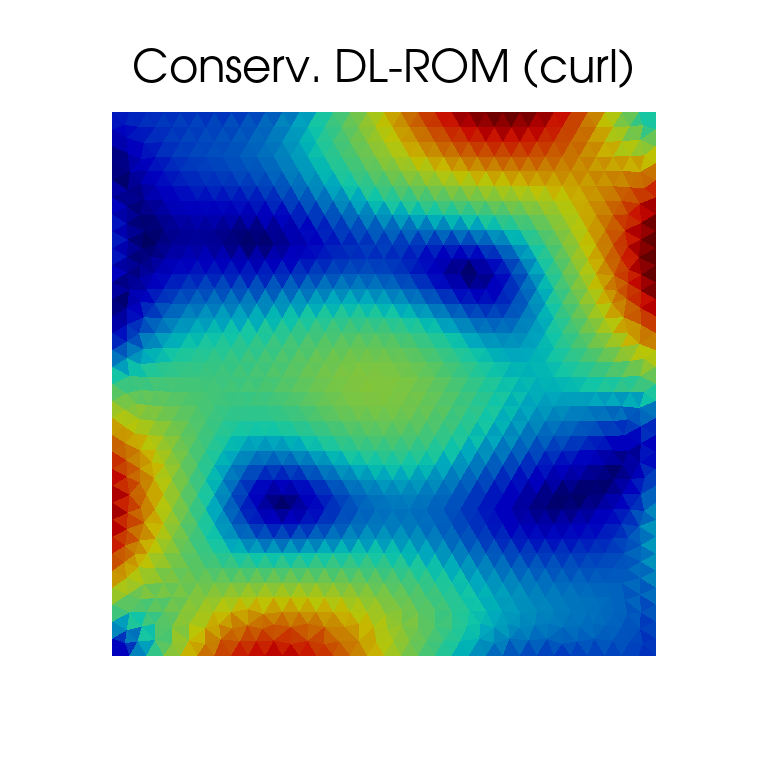}
    \includegraphics[width=0.3\textwidth]{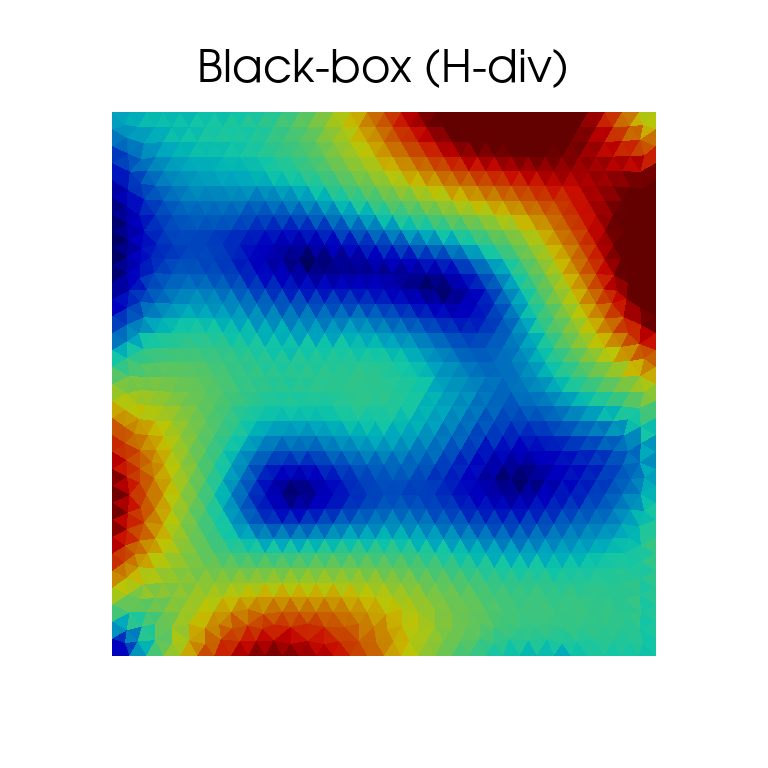}
    \includegraphics[height=3.45cm]{fig/case1/case1flow-bar.png}
    \\\vspace{-0.5cm}
    \includegraphics[width=0.3\textwidth]{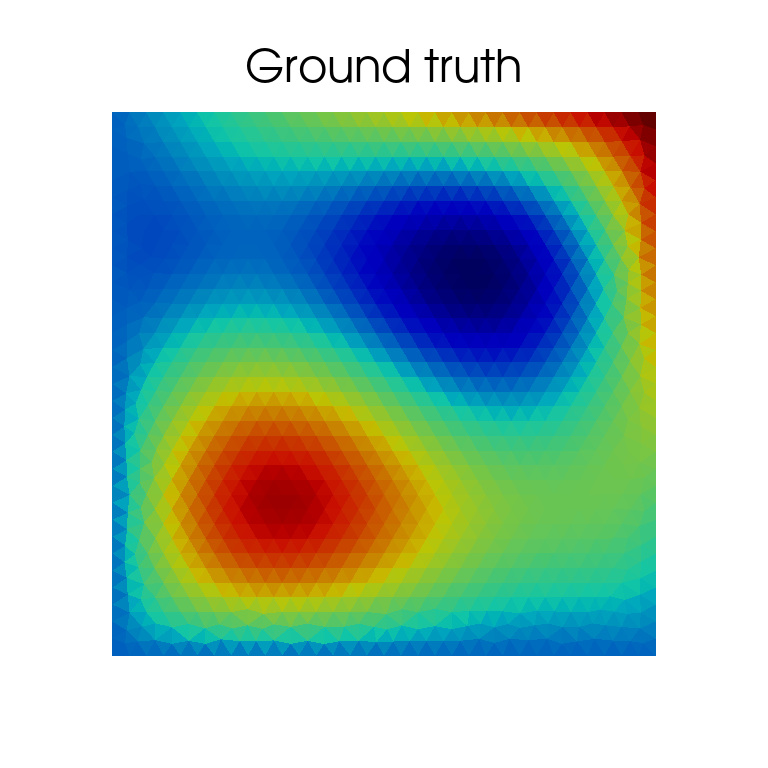}
    \includegraphics[width=0.3\textwidth]
    {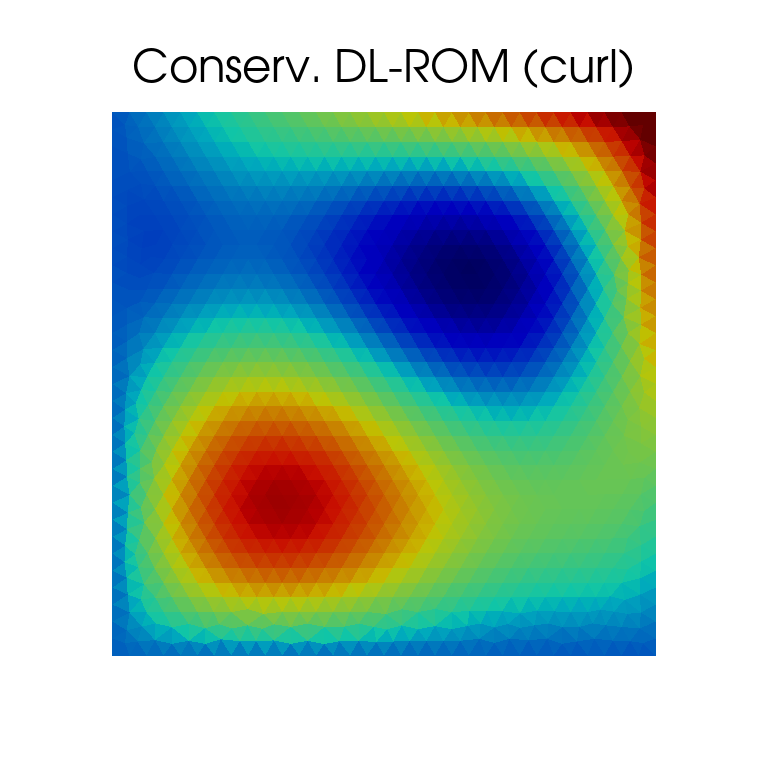}
    \includegraphics[width=0.3\textwidth]{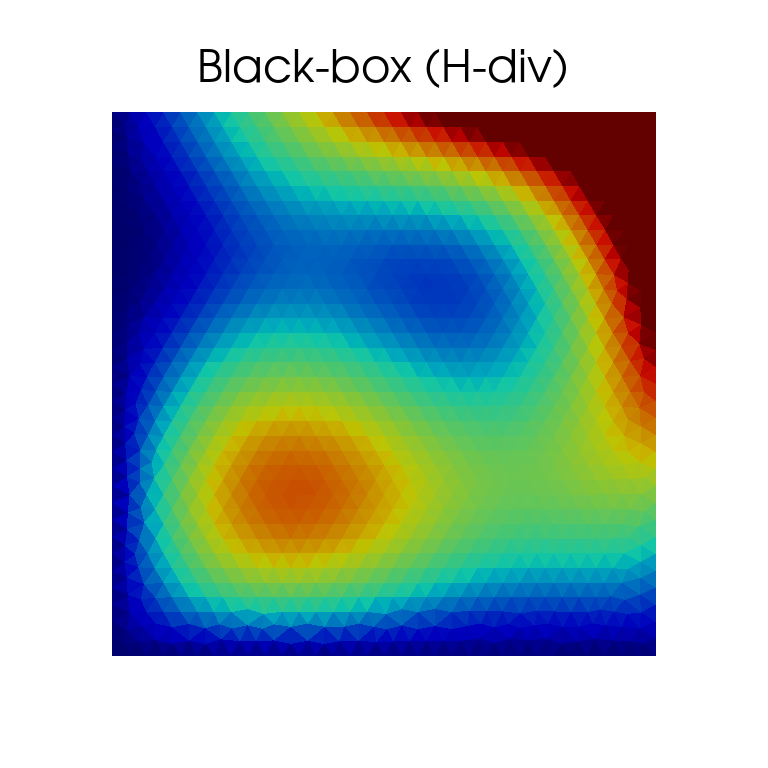}
    \includegraphics[height=3.45cm]{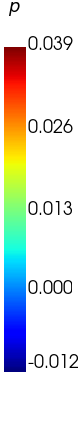}
    \caption{Magnitude of the velocity fields $|\bq|$ (top) and pressure fields $p$ (bottom) for Case 3, \Cref{subsec:exp:sines}. FOM solution vs ROM approximations for an unseen configuration of the problem parameters, $\bmu= [0.388, 0.040, -0.193, 0.744].$}
    \label{fig:case3all}
\end{figure}
\begin{table}
    \label{tab:case3errors}
    \centering
    \caption{Models comparison for the third case study, \Cref{subsec:exp:nonlinear}. Table entries read as in Table \ref{tab:exp:sines}.}
    \begin{tabular}{lllllr}
    \hline
     \textbf{Model} & \textbf{Map onto}\;\;& \textbf{$L^{2}$ error} & \textbf{$H(\vdiv)$ error} & \textbf{$L^{2}$ error}&\textbf{Training time}\vspace{-0.4em}\\
     &\textbf{kernel} $S_{0}$\;\;&Flow $\bq$&Flow $\bq$&Pressure $p$&\\
     \hline\hline
     Black-box $L^{2}$& None & \;\;\;2.71\% &\;\;\;\;\;\;\;\;\;\;\;4.03\% &\hspace{1.5em}3.47\%& 3m 52.5s\\
     Black-box $H(\vdiv)$& None & \;13.45\% &\;\;\;\;\;\;\;\;\;\;\;3.12\% & \;\;\;\;20.18\% & 4m 42.4s\\
     Conserv. POD-NN & $V_n$ & \;\;\;2.67\% &\;\;\;\;\;\;\;\;\;\;\;0.64\% & \hspace{1.5em}4.03\%& 2m 32.1s\\
     \rowcolor{Gray} Conserv. DL-ROM & $\curl$ & \;\;\;2.31\%&\;\;\;\;\;\;\;\;\;\;\;0.63\%&\hspace{1.5em}2.80\%& 4m 19.7s\\
     Conserv. DL-ROM & $I-S_{I}B$ & \;\;\;2.83\% &\;\;\;\;\;\;\;\;\;\;\;0.87\% &\hspace{1.5em}3.29\% & 7m 14.1s\\\hline\\
    \end{tabular}
    \label{tab:exp:nonlinear}
\end{table}
\begin{figure}
    \centering
    \includegraphics[width=\textwidth]{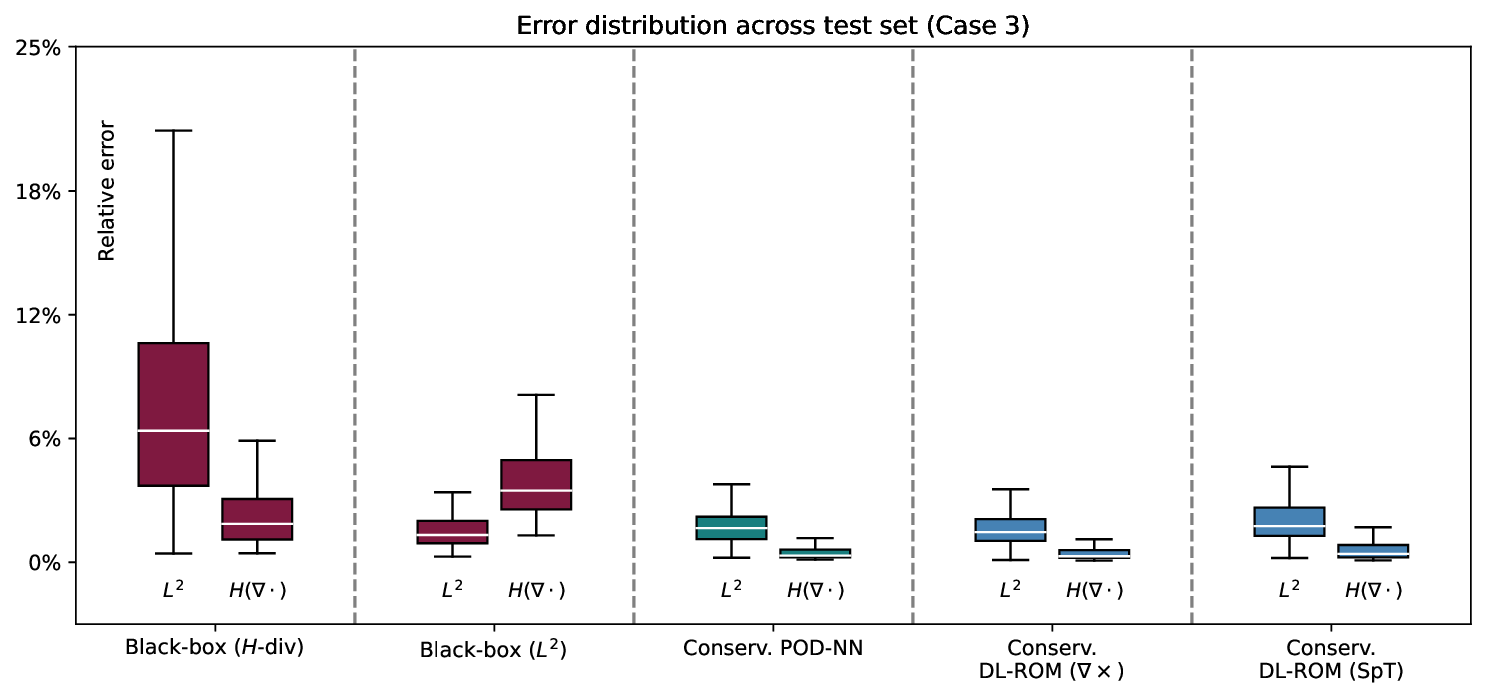}
    \caption{Flux error distributions for the third case study, \Cref{subsec:exp:nonlinear}.}
    \label{fig:case3boxplots}
\end{figure}

\section{Concluding remarks}
\label{sec: Concluding remarks}

We proposed an innovative Deep Learning strategy for the model order reduction of Darcy flow systems, in the context of partial differential equations with linear constraints. The methodology hinges on the idea of splitting the flux into homogeneous and particular solutions, employing a neural network architecture and a kernel projector for the approximation of the former, and an efficient spanning tree algorithm for the latter.

Leveraging on this idea, we proposed three
model order reduction strategies,
all adhering to linear constraints and based upon neural networks: 
\begin{itemize}
	\item[a)] \emph{Conservative POD-NN}: an approach based on a data-driven potential subspace; 
	\item[b)] \emph{Curl DL-ROM}: a technique based on the theory of differential complexes;
	\item[c)] \emph{SpT DL-ROM}: a strategy that relies entirely on an efficient spanning tree solve.
\end{itemize}
To substantiate the efficacy of the proposed strategies, we compared their performances with those attained by naive DNN regression (or "black-box" approaches). This comparative analysis was carried out across three different test cases: a 
2D problem with a non-trivial dependency on the model parameters, a mixed-dimensional model of flow in fractured porous medium, and a non-linear Darcy-Forchheimer flow system.


Our numerical experiments show that all conservative approaches can successfully incorporate the linear constraint, thus overcoming the mass conservation issues arising in black-box approaches. This is made possible by our construction with a potential space and, as the experiments show, the same result cannot be easily achieved by simply including the linear constraint within the loss function. In fact, as testified by the poor performance of the $H(\vdiv)$ black-box surrogate, the inclusion of a divergence term does not invariably yield superior results. Instead, it may actually impair ROM performance. 


We moreover observe that none of the three approaches, conservative POD-NN, Curl DL-ROM and SpT DL-ROM, is unquestionably better than the others for all problems. The two DL-ROM techniques are better suited if the PDE features either strong nonlinearities or space-varying parameters. Otherwise, the conservative POD-NN can be a good fit, with high accuracy and reduced training times. It is noteworthy that although our construction was focused on providing a reliable approximation of the flux, the three approaches were still capable of retrieving a good approximation of the pressure fields as well. In general, the quality of such an approximation was better for ROMs that were sufficiently $L^{2}$-accurate on the flux. In this context, although our findings are promising, there might still be room for improvement.



Our proposal stands as a data-driven, but physical, non-intrusive alternative to other well-established model order reduction techniques, such as the Reduced Basis method. In particular, it can be of high interest whenever intrusive strategies are unavailable, possibly due to high computational costs or code inaccessibility, or when projection-based methods encounter substantial difficulties, e.g. in case of slow decay in the Kolmogorov $n$-width. 
Future research may explore mechanisms to embed deeper physical understanding into the models, ensuring adherence to inherent physical laws and constraints, and investigating the integration of deep learning techniques with traditional numerical schemes across various computational domains. 







\bibliographystyle{spmpsci}
\bibliography{references}

\section*{Declarations}

The authors have no relevant financial or non-financial interests to disclose.
The datasets and source code generated and analyzed during the current study are available in the repository \url{https://github.com/compgeo-mox/conservative_ml}.

\appendix
\captionsetup{font={small}}
\section{Neural network architectures}
\label{sec: Appendix - Architectures}

In this Section, we report all the technical details about the neural network architectures employed across the three experiments. We mention that all DNN architectures were constructed using dense layers, primarily relying on the 0.1-leakyReLU activation. All networks were trained using the L-BFGS optimizer with standard parameter values (learning rate = 1) and for a total of 500 epochs.

We recall that a dense layer with input dimension $m_{1}$, output dimension $m_{2}$ and activation $\rho:\mathbb{R}\to\mathbb{R}$ is a nonlinear map from $\mathbb{R}^{m_{1}}\to\mathbb{R}^{m_{2}}$ of the form
$$\mathbf{v}\mapsto \rho\left(\mathbf{W}\mathbf{v}+\mathbf{b}\right)$$
where $\mathbf{W}\in\mathbb{R}^{m_{2}\times m_{1}}$ is the weight matrix, $\mathbf{b}\in\mathbb{R}^{m_{2}}$ the bias vector and, with little abuse of notation, $$\rho([w_{1},\dots,w_{m}]^{T}):=[\rho(w_{1}),\dots,\rho(w_{m})]^{T}.$$
Then, classical deep feed forward neural networks (such as those employed here) are obtained via sequential composition of multiple layers.

N.B.: in what follows, $\rho$ is always assumed to be the 0.1-leakyReLU activation, namely, $$\rho:\;\;x\mapsto 0.1x\mathbf{1}_{(-\infty,0)}(x)+x\mathbf{1}_{[0,+\infty)}(x).$$

\subsection{Architectures for Case 1}
For this test case, all architectures employ a \textit{feature layer}, whose purpose is to provide a preliminary preprocessing of the input parameters. 
The latter consists of a non-learnable map $F:\mathbb{R}^{2}\to\mathbb{R}^{225}$ acting as
$$[\mu_{0},\mu_{1}]\mapsto \left[\sin\left(2\pi\mu_{0}\frac{\text{mod}(j,15)}{14}\right)\sin\left(2\pi\mu_{1}\frac{j-\text{mod}(j,15)}{14\cdot15}\right)\right]_{j=1}^{225}.$$
That is, $F$ maps the two parameters $[\mu_{0},\mu_{1}]$ onto a discrete representation of the source term $f^{\mu_{0},\mu_{1}}(x_{0},x_{1})=\sin(\mu_{0}2\pi x_{0})\sin(\mu_{1}2\pi x_{1})$, obtained by evaluating the latter over a uniform $15\times15$ grid. As the PDE solution only depends on $[\mu_{0},\mu_{1}]$ through $f^{\mu_{0},\mu_{1}}$, this transformation can be seen as a natural preprocessing of the input data.\\

\setcounter{table}{0}
\captionsetup[table]{skip=0pt,singlelinecheck=off}
\begin{table}[ht!]
    \centering
    \caption{Potential network for the Conservative POD-NN approach. The output dimension coincides with the POD space dimension, here $n=100.$ The first layer is not learnable and coincides with the feature map $F$.\vspace{-0.5em}}
    \begin{tabular}{l}
         \hspace{\textwidth}
    \end{tabular}
    \centering
    \begin{tabular}{llll}
    \hline\hline
    \textbf{Layer}  &  \textbf{Input} &  \textbf{Output} & \textbf{Activation}\\\hline
    Feature extractor & 2  & 225 & -\\
    Dense & 225  & 100 & $\rho$\\
    Dense & 100  & 100 & $\rho$\\
    Dense & 100  & 100 & -
    \\\hline\hline
    \end{tabular}
\end{table}

\begin{table}[ht!]
\caption{Architectures of the Conservative DL-ROM approach (both versions). Potential network on the left (above dashed line = $\phi$, below dashed line = $\Psi$), auxiliary encoder module on the right. Here, $\dim(\bR)$ equals either 1265 or 3664 depending on the ROM variation, that is, Curl and SpT, respectively.\vspace{-0.5em}\\} 
    \begin{tabular}{clcc}
    Potential network $\mathcal{N}=\Psi\circ\phi$ & \hspace{3.25cm} & Encoder $\Psi'$ & \hspace{3.25cm}
    \end{tabular}
    \begin{minipage}{0.49\textwidth}
    \begin{tabular}{llll}
    \hline\hline
    \textbf{Layer}  &  \textbf{Input} &  \textbf{Output} & \textbf{Activ.}\\\hline
    Feature extractor & 2  & 225 & -\\
    Dense & 225  & 100 & $\rho$\\\hdashline
    Dense & 100  & 200 & $\rho$\\
    Dense & 200  & $\dim(\bR)$ & -
    \\\hline\hline
    \end{tabular}
    \end{minipage}\hfill
    \begin{minipage}{0.49\textwidth}
     \begin{tabular}{llll}
    \hline\hline
    \textbf{Layer}  &  \textbf{Input} &  \textbf{Output} & \textbf{Activ.}\\\hline
    Dense & 3664  & 100 & $\rho$\\
    Dense & 100  & 100 & $\rho$
    \\\hline\hline\\\\
    \end{tabular}
    \end{minipage}
\end{table}

\begin{table}[ht!]
    \centering
    \caption{Architecture for the black-box surrogates.\vspace{-0.5em}}\begin{tabular}{c}
         \hspace{\textwidth}
    \end{tabular}
    \begin{tabular}{llll}
    \hline\hline
    \textbf{Layer}  &  \textbf{Input} &  \textbf{Output} & \textbf{Activation}\\\hline
    Feature extractor & 2  & 225 & -\\
    Dense & 225  & 100 & $\rho$\\
    Dense & 100  & 200 & $\rho$\\
    Dense & 200  & 500 & $\rho$\\
    Dense & 500  & 3664 & -
    \\\hline\hline
    \end{tabular}    
\end{table}

\newpage
\subsection{Architectures for Case 2}
We report the architectures for the 2nd test case. We recall that the problem features $5$ scalar parameters, while the FOM dimension (for the flux field) is $17249$.

\begin{table}[ht!]
\caption{Conservative POD-NN, potential network. Output dimension = POD space dimension, here $n=10.$\vspace{-0.5em}}
    \begin{tabular}{l}
         \hspace{\textwidth}
    \end{tabular}
    \centering
    \begin{tabular}{llll}
    \hline\hline
    \textbf{Layer}  &  \textbf{Input} &  \textbf{Output} & \textbf{Activation}\\\hline
    Dense & 5 & 200 & $\rho$\\
    Dense & 200  & 100 & $\rho$\\
    Dense & 100  & 10 & -
    \\\hline\hline
    \end{tabular}\vspace{-1em}
    
\end{table}

\begin{table}[ht!]
\caption{Architectures of the Conservative DL-ROM approach (both versions). Potential network on the left (above dashed line = $\phi$, below dashed line = $\Psi$), auxiliary encoder module on the right.
    Here, $\dim(\bR)$ equals either 11182 or 17249 depending on the ROM variation, that is, Curl and SpT, respectively.\vspace{-0.5em}\\} 
    \begin{tabular}{clcc}
    Potential network $\mathcal{N}=\Psi\circ\phi$ & \hspace{3.25cm} & Encoder $\Psi'$ & \hspace{3.25cm}
    \end{tabular}
    \begin{minipage}{0.49\textwidth}
    \begin{tabular}{llll}
    \hline\hline
    \textbf{Layer}  &  \textbf{Input} &  \textbf{Output} & \textbf{Activ.}\\\hline
    Dense & 5  & 200 & $\rho$\\
    Dense & 200  & 10 & $\rho$\\\hdashline
    Dense & 10   & 500 & $\rho$\\
    Dense & 500  & $\dim(\bR)$ & -
    \\\hline\hline
    \end{tabular}\end{minipage}
    \hfill\begin{minipage}{0.49\textwidth}
     \begin{tabular}{llll}
    \hline\hline
    \textbf{Layer}  &  \textbf{Input} &  \textbf{Output} & \textbf{Activ.}\\\hline
    Dense & 17249  & 10 & $\rho$
    \\\hline\hline\\\\\\
    \end{tabular}
    \end{minipage}
\end{table}

\begin{table}[ht!]
    \centering
    \caption{Architecture for the black-box surrogates.\vspace{-0.5em}}\begin{tabular}{c}
         \hspace{\textwidth}
    \end{tabular}
    \centering
    \begin{tabular}{llll}
    \hline\hline
    \textbf{Layer}  &  \textbf{Input} &  \textbf{Output} & \textbf{Activation}\\\hline
    Dense & 5    & 200 & $\rho$\\
    Dense & 200  & 10  & $\rho$\\
    Dense & 100  & 500 & $\rho$\\
    Dense & 500  & 17249 & -
    \\\hline\hline
    \end{tabular}
\end{table}

\newpage
\subsection{Architectures for Case 3}
We report the architectures for the last test case (nonlinear flow). We recall that the problem features $4$ parameters, while the FOM dimension (for the flux) is $3664.$

\begin{table}[ht!]
\caption{Potential network for the Conservative POD-NN approach.\vspace{-0.5em}}
    \begin{tabular}{l}
         \hspace{\textwidth}
    \end{tabular}
    \centering
    \begin{tabular}{llll}
    \hline\hline
    \textbf{Layer}  &  \textbf{Input} &  \textbf{Output} & \textbf{Activation}\\\hline
    Dense & 4 & 50 & $\rho$\\
    Dense & 50  & 50 & $\rho$\\
    Dense & 50  & 100 & $\rho$\\
    Dense & 100  & 4 & -
    \\\hline\hline
    \end{tabular}
\end{table}

\begin{table}[ht!]
   \caption{Architectures of the Conservative DL-ROM approach (both versions). Potential network on the left (above dashed line = $\phi$, below dashed line = $\Psi$), encoder module on the right. Here,
    $\dim(\bR)=11182,17249$ for the Curl and SpT variations, respectively.\vspace{-0.5em}\\} 
    \begin{tabular}{clcc}
    Potential network $\mathcal{N}=\Psi\circ\phi$ & \hspace{3.25cm} & Encoder $\Psi'$ & \hspace{3.25cm}
    \end{tabular}
    \begin{minipage}{0.49\textwidth}
    \begin{tabular}{llll}
    \hline\hline
    \textbf{Layer}  &  \textbf{Input} &  \textbf{Output} & \textbf{Activ.}\\\hline
    Dense & 4  & 50 & $\rho$\\
    Dense & 50  & 50 & $\rho$\\
    Dense & 50  & 4 & $\rho$\\\hdashline
    Dense & 4   & 50 & $\rho$\\
    Dense & 50  & $\dim(\bR)$ & -
    \\\hline\hline
    \end{tabular}\end{minipage}
    \hfill
    \begin{minipage}{0.49\textwidth}
     \begin{tabular}{llll}
    \hline\hline
    \textbf{Layer}  &  \textbf{Input} &  \textbf{Output} & \textbf{Activ.}\\\hline
    Dense & 3664  & 10 & $\rho$
    \\\hline\hline\\\\\\\\
    \end{tabular}
    \end{minipage}
\end{table}

\begin{center}
\begin{table}[ht!]
    \centering
    \caption{\centering Architecture for the black-box surrogates.\vspace{-0.5em}}
    \begin{tabular}{c}     \hspace{\textwidth}
    \end{tabular}
    \begin{tabular}{llll}
    \hline\hline
    \textbf{Layer}  &  \textbf{Input} &  \textbf{Output} & \textbf{Activation}\\\hline
    Dense & 4    & 50 & $\rho$\\
    Dense & 50  & 50  & $\rho$\\
    Dense & 50  & 4 & $\rho$\\
    Dense & 4  & 50 & $\rho$\\
    Dense & 50  & 3644 & -
    \\\hline\hline
    \end{tabular}
\end{table}
\end{center}

\section{Mixed-dimensional model for flow in fractured porous media}
\label{sec: Mixed-dimensional model for flow in fractured porous media}

We give a concise exposition of the fracture flow model used in \Cref{subsec:exp:fractures}. The interested reader is referred to \cite{berre2021verification} for more detail concerning the model, \cite{boon2018robust} for the mixed finite element method, and \cite{boon2021functional} for the admissible geometries and mixed-dimensional function spaces.

Let us decompose the domain into subdomains $\Omega_i$ with $d_i$ its dimension and $i$ from the index set $\cI$. Fractures will thus be given by subdomains with $d_i = 2$, fracture intersections are one-dimensional with $d_i = 1$ and the subdomains with $d_i = 0$ represent the intersection points. 
For given $\Omega_i$, let $\partial_j \Omega$ with $d_j = d_i - 1$ be the part of its boundary that coincides geometrically with a lower-dimensional neighbor $\Omega_j$. Let $\bnu_i$ on $\partial_j \Omega$ be the unit normal vector, oriented outward with respect to $\Omega_i$.

We continue with the material parameters. Let $\varepsilon_i$ be a given length constant that represents the fracture aperture if $d_i = 2$. Moreover, $\varepsilon_i^2$ represents the cross-sectional area of intersection lines with $d_i = 1$, and $\varepsilon_i^3$ represents the volume of intersection points for $d_i = 0$. Next, for given fracture conductivity $K_{frac}$, we define the effective conductivities:
\begin{align}
	K_{\parallel, i} &:= \varepsilon_i^{3 - d_i} K_{frac} \text{ on }\Omega_i, & 
	K_{\perp, ij} &:= \varepsilon_i^{3 - d_i}\frac{2}{\varepsilon_j} K_{frac} \text{ on }\partial_j \Omega_i.
\end{align}

For ease of notation, we collect the flux and pressure variables into mixed-dimensional variables, which we denote using a Gothic font:
\begin{align}
    \fp &:= \bigoplus_{\substack{i \in \cI \\ 0 \le d_i \le 3}} p_i, &
    \fq &:= \bigoplus_{\substack{i \in \cI \\ 1 \le d_i \le 3}} q_i.
\end{align}

The fracture flow model, cf. \cite{boon2018robust,berre2021verification}, then takes the form \eqref{eq:darcy-type} with
\begin{subequations}
\begin{align}
	\langle A \fq, \tilde \fq \rangle 
	&:= \sum_{\substack{i \in \cI \\ 1 \le d_i \le 3}} 
		(K_{\parallel, i}^{-1} \bq_i, \tilde \bq_i)_{\Omega_i}
	 + \sum_{\substack{j \in \cI \\ d_j = d_i - 1}} 
	 	(K_{\perp, ij}^{-1} \bnu_i \cdot \bq_i, \bnu_i \cdot \tilde \bq_i)_{\partial_j \Omega_i} \\
	\langle B \fq, \tilde \fp \rangle 
	&:= \sum_{\substack{i \in \cI \\ 1 \le d_i \le 3}} 
		(\div \bq_i, \tilde p_i)_{\Omega_i}
	 - \sum_{\substack{j \in \cI \\ d_j = d_i - 1}} 
	 	(\bnu_i \cdot \bq_i, \tilde p_j)_{\partial_j \Omega_i}
\end{align}
\end{subequations}

The operator $B$ corresponds to the mixed-dimensional divergence operator which we denote by $\fD \cdot \fq$. This operator is part of a differential complex \cite{boon2021functional} and the mixed-dimensional curl $\fD \times$ has the property $\fD \cdot (\fD \times \fr) = 0$ for all $\fr$ in its domain. Precise definitions of the mixed-dimensional curl can be found in \cite{boon2021functional,boon2023reduced,budisa2020mixed}.

To conclude this section, we define the norms on the mixed-dimensional variables as follows:
\begin{subequations}
\begin{align}
	\| \fp \|_P^2
	&:= \sum_{i \in \cI} \| p_i \|_{\Omega_i}^2 \\
	\| \fq \|_Q^2
	&:= \left(\sum_{\substack{i \in \cI \\ 1 \le d_i \le 3}} 
	\| \bq_i \|_{\Omega_i}^2 
	+ \sum_{\substack{j \in \cI \\ d_j = d_i - 1}} 
	\| \bnu_i \cdot \bq_i \|_{\partial_j \Omega_i}^2 \right)
	+ \| \fD \cdot \fq \|_P^2.
\end{align}
\end{subequations}
With a slight abuse of notation, we refer to these norms as $L^2$ and $H(\vdiv)$, respectively, in \Cref{subsec:exp:fractures}.

\end{document}